\documentclass[11pt]{amsart}
\usepackage{amsopn}
\usepackage{amssymb, amscd}
\usepackage{multirow}
\usepackage{enumerate}

\usepackage{hyperref}

\newcommand{\nc}{\newcommand}

\nc{\fg}{\mathfrak{f} } \nc{\vg}{\mathfrak{v} } \nc{\wg}{\mathfrak{w} }
\nc{\zg}{\mathfrak{z} } \nc{\ngo}{\mathfrak{n} } \nc{\kg}{\mathfrak{k} }
\nc{\mg}{\mathfrak{m} } \nc{\bg}{\mathfrak{b} } \nc{\ggo}{\mathfrak{g} }
\nc{\ggob}{\overline{\mathfrak{g}} } \nc{\sog}{\mathfrak{so} }
\nc{\sug}{\mathfrak{su} } \nc{\spg}{\mathfrak{sp} } \nc{\slg}{\mathfrak{sl} }
\nc{\glg}{\mathfrak{gl} } \nc{\cg}{\mathfrak{c} } \nc{\rg}{\mathfrak{r} }
\nc{\hg}{\mathfrak{h} } \nc{\tg}{\mathfrak{t} } \nc{\ug}{\mathfrak{u} }
\nc{\dg}{\mathfrak{d} } \nc{\ag}{\mathfrak{a} } \nc{\pg}{\mathfrak{p} }
\nc{\sg}{\mathfrak{s} } \nc{\affg}{\mathfrak{aff} } \nc{\qg}{\mathfrak{q} }

\nc{\pca}{\mathcal{P}} \nc{\nca}{\mathcal{N}} \nc{\lca}{\mathcal{L}}
\nc{\oca}{\mathcal{O}} \nc{\mca}{\mathcal{M}} \nc{\tca}{\mathcal{T}}
\nc{\aca}{\mathcal{A}} \nc{\cca}{\mathcal{C}} \nc{\gca}{\mathcal{G}}
\nc{\sca}{\mathcal{S}} \nc{\hca}{\mathcal{H}} \nc{\bca}{\mathcal{B}}
\nc{\dca}{\mathcal{D}} \nc{\ica}{\mathcal{I}} \nc{\val}{\operatorname{val}}

\nc{\vp}{\varphi} \nc{\ddt}{\frac{d}{dt}} \nc{\dds}{\frac{d}{ds}}
\nc{\dpar}{\frac{\partial}{\partial t}} \nc{\im}{\mathrm{i}}

\nc{\SO}{\mathrm{SO}} \nc{\Spe}{\mathrm{Sp}} \nc{\Sl}{\mathrm{SL}}
\nc{\SU}{\mathrm{SU}} \nc{\Or}{\mathrm{O}} \nc{\U}{\mathrm{U}} \nc{\Gl}{\mathrm{GL}}
\nc{\Se}{\mathrm{S}} \nc{\Cl}{\mathrm{Cl}} \nc{\Spein}{\mathrm{Spin}}
\nc{\Pin}{\mathrm{Pin}} \nc{\G}{\mathrm{GL}_n(\RR)} \nc{\g}{\mathfrak{gl}_n(\RR)}

\nc{\RR}{{\Bbb R}} \nc{\HH}{{\Bbb H}} \nc{\CC}{{\Bbb C}} \nc{\ZZ}{{\Bbb Z}}
\nc{\FF}{{\Bbb F}} \nc{\NN}{{\Bbb N}} \nc{\QQ}{{\Bbb Q}} \nc{\PP}{{\Bbb P}} \nc{\OO}{{\Bbb O}}

\nc{\vs}{\vspace{.2cm}} \nc{\vsp}{\vspace{1cm}} \nc{\ip}{\langle\cdot,\cdot\rangle}
\nc{\ipp}{(\cdot,\cdot)} \nc{\la}{\langle} \nc{\ra}{\rangle} \nc{\unm}{\frac{1}{2}}
\nc{\unc}{\frac{1}{4}} \nc{\und}{\frac{1}{16}} \nc{\no}{\vs\noindent}
\nc{\lam}{\Lambda^2(\RR^n)^*\otimes\RR^n} \nc{\tangz}{{\rm T}^{\rm Zar}}
\nc{\nor}{{\sf n}}  \nc{\mum}{/\!\!/} \nc{\kir}{/\!\!/\!\!/}
\nc{\Ri}{\tfrac{4\Ric_{\mu}}{||\mu||^2}} \nc{\ds}{\displaystyle}
\nc{\ben}{\begin{enumerate}} \nc{\een}{\end{enumerate}} \nc{\f}{\frac}
\nc{\lb}{[\cdot,\cdot]} \nc{\isn}{\tfrac{1}{||v||^2}}
\nc{\gkp}{(\ggo=\kg\oplus\pg,\ip)} \nc{\ukh}{(\ug=\kg\oplus\hg,\ip)}
\nc{\tgkp}{(\tilde{\ggo}=\kg\oplus\pg,\ip)}
\nc{\wt}{\widetilde} \nc{\mm}{M}
\nc{\iop}{\mathtt{i}} \nc{\jop}{\mathtt{j}}

\nc{\Hess}{\operatorname{Hess}} \nc{\ad}{\operatorname{ad}}
\nc{\Ad}{\operatorname{Ad}} \nc{\rank}{\operatorname{rank}}
\nc{\Irr}{\operatorname{Irr}} \nc{\End}{\operatorname{End}}
\nc{\Aut}{\operatorname{Aut}} \nc{\Inn}{\operatorname{Inn}}
\nc{\Der}{\operatorname{Der}} \nc{\Ker}{\operatorname{Ker}}
\nc{\Iso}{\operatorname{Iso}} \nc{\Diff}{\operatorname{Diff}}
\nc{\Lie}{\operatorname{Lie}} \nc{\tr}{\operatorname{tr}} \nc{\dif}{\operatorname{d}}
\nc{\sen}{\operatorname{sen}} \nc{\modu}{\operatorname{mod}}
\nc{\CRic}{\operatorname{PP}} \nc{\Cric}{\operatorname{P}} \nc{\Ricci}{\operatorname{Ric}}
\nc{\sym}{\operatorname{sym}} \nc{\herm}{\operatorname{herm}} \nc{\symac}{\operatorname{sym^{ac}}}
\nc{\symc}{\operatorname{sym^{c}}} \nc{\scalar}{\operatorname{scal}}
\nc{\grad}{\operatorname{grad}} \nc{\ricci}{\operatorname{Rc}}
\nc{\Nor}{\operatorname{Norm}}  \nc{\ricc}{\operatorname{Rc^{c}}}
\nc{\Ricc}{\operatorname{Ric^{c}}} \nc{\ricac}{\operatorname{Rc^{ac}}}
\nc{\Ricac}{\operatorname{Ric^{ac}}} \nc{\Riem}{\operatorname{R}}
\nc{\riccig}{\operatorname{ric^{\gamma}}} \nc{\Rin}{\operatorname{M}}
\nc{\Le}{\operatorname{L}} \nc{\tang}{\operatorname{T}}
\nc{\level}{\operatorname{level}} \nc{\rad}{\operatorname{r}}
\nc{\abel}{\operatorname{ab}} \nc{\CH}{\operatorname{CH}}
\nc{\mcc}{\operatorname{mcc}} \nc{\Adj}{\operatorname{Adj}}
\nc{\Order}{\operatorname{O}}  \nc{\inj}{\operatorname{inj}} \nc{\proy}{\operatorname{proy}}
\nc{\vol}{\operatorname{vol}} \nc{\Diag}{\operatorname{Dg}}
\nc{\Spec}{\operatorname{Spec}} \nc{\Ima}{\operatorname{Im}} \nc{\Rea}{\operatorname{Re}}
\nc{\spann}{\operatorname{sp}}

\theoremstyle{plain}
\newtheorem{theorem}{Theorem}[section]
\newtheorem{proposition}[theorem]{Proposition}
\newtheorem{corollary}[theorem]{Corollary}

\theoremstyle{definition}

\theoremstyle{remark}
\newtheorem{remark}[theorem]{Remark}

\newtheorem{example}[theorem]{Example}

\title{The classification of ERP $G_2$-structures on Lie groups}

\author{Jorge Lauret} \author{Marina Nicolini}

\address{Universidad Nacional de C\'ordoba, FaMAF and CIEM, 5000 C\'ordoba, Argentina}
\email{lauret@famaf.unc.edu.ar} \email{mnicolini@famaf.unc.edu.ar}

\thanks{This research was partially supported by grants from FONCYT and Universidad Nacional de C\'ordoba.}

\begin{document}

\maketitle

\begin{abstract}
A complete classification of left-invariant closed $G_2$-structures on Lie groups which are extremally Ricci pinched (i.e.\ $d\tau = \tfrac{1}{6}|\tau|^2\vp + \tfrac{1}{6}\ast(\tau\wedge\tau)$), up to equivalence and scaling, is obtained.  There are five of them, they are defined on five different completely solvable Lie groups and the $G_2$-structure is exact in all cases except one, given by the only example in which the Lie group is unimodular.
\end{abstract}


\section{Introduction}\label{intro}

A $G_2$-{\it structure} on a $7$-dimensional differentiable manifold $M$ is a differential $3$-form $\vp$ on $M$ which is positive (or definite), in the sense that $\vp$ (uniquely) determines a Riemannian metric $g$ on $M$ together with an orientation.  Playing a role analogous in a way to that of almost-K\"ahler structures in almost-hermitian geometry, closed $G_2$-structures (i.e.\ $d\vp=0$) have been studied by several authors (see e.g.\ \cite{Bry,ClyIvn1, ClyIvn2,FrnFinRff,LtyWei,PdsRff}).  They appear as natural candidates to be deformed via the Laplacian flow $\dpar\vp(t) = \Delta\vp(t)$ toward a {\it torsion-free} $G_2$-structure (i.e.\ $d\vp=0$ and $d\ast\vp=0$) producing a Ricci flat Riemannian metric with holonomy contained in $G_2$ (see \cite{Lty} for an account of recent advances).  The torsion of a closed $G_2$-structure $\vp$ is completely determined by the $2$-form $\tau=-\ast d\ast\vp$, which in addition satisfies that $d\ast\vp=\tau\wedge\vp$.

The following remarkable curvature estimate for closed $G_2$-structures on a compact manifold $M$ was discovered by Bryant (see \cite[Corollary 3]{Bry}):
\begin{equation}\label{estB}
\int_M \scalar^2 \ast 1 \leq 3\int_M |\Ricci|^2 \ast 1,
\end{equation}
where $\scalar$ and $\Ricci$ respectively denote scalar curvature and Ricci tensor of $(M,g)$.  Bryant called {\it extremally Ricci-pinched} (ERP for short) the structures at which equality holds in \eqref{estB} (see \cite[Remark 13]{Bry}) and proved that they are characterized by the following neat equation,
\begin{equation}\label{ERPeq}
d\tau = \tfrac{1}{6}|\tau|^2\vp + \tfrac{1}{6}\ast(\tau\wedge\tau).
\end{equation}
In the compact case, this is actually the only way in which $d\tau$ can quadratically depend on $\tau$ (see \cite[(4.66)]{Bry}).  A non-necessarily compact $(M,\vp)$ satisfying \eqref{ERPeq} is also called ERP.

The first examples of ERP $G_2$-structures in the literature are {\it homogeneous} (i.e.\ the automorphism group $\Aut(M,\vp):=\{ f\in\Diff(M):f^*\vp=\vp\}$ acts transitively on $M$).  In \cite{Bry}, Bryant gave the first example on the homogeneous space $\Sl_2(\CC)\ltimes\CC^2/\SU(2)$, which indeed admits a compact locally homogeneous quotient (the only compact ERP structure known so far), a second one was found on a unimodular solvable Lie group in \cite{LS-ERP} and a curve was given in \cite{FinRff2} (each structure in the curve is actually equivalent to Bryant's example, although the Lie groups involved are pairwise non-isomorphic).  Three more examples were recently given on non-unimodular solvable Lie groups in \cite{ERP}.  We refer to \cite{Bll} for examples of ERP $G_2$-structures which are not homogeneous.

It is worth pointing out that homogeneous $G_2$-geometry include the following particular features:
\begin{enumerate}[{\small $\bullet$} ]
\item Torsion-free $G_2$-structures are necessarily flat by \cite{AlkKml}.

\item Closed $G_2$-structures are only allowed on non-compact manifolds by \cite{PdsRff}.

\item Einstein closed $G_2$-structures do not exist by \cite{FrnFinMnr,ArrLfn}.

\item The only other possibility for a quadratic dependence is to have $d\tau = \frac{1}{7}|\tau|^2\vp$, whose existence  is an open problem (see \cite{LS-ERP}).

\item Estimate \eqref{estB} does not hold in general, examples with $\scalar^2 >  3 |\Ricci|^2$ were found in \cite{LS-ERP}.
\end{enumerate}

In this paper, we continue the study of left-invariant ERP $G_2$-structures on Lie groups initiated in \cite{ERP}, where it was shown that the ERP condition \eqref{ERPeq} requires very strong structure constraints on the Lie algebra.  By using such a structure theorem (see Theorem \ref{main}) as a starting point, we have obtained a complete classification.

As usual, two manifolds endowed with $G_2$-structures $(M,\vp)$ and $(M',\vp')$ are called {\it equivalent} if there exists a diffeomorphism $f:M\longrightarrow M'$ such that $\vp=f^*\vp'$.  Two Lie groups endowed with left-invariant $G_2$-structures $(G,\vp)$ and $(G',\vp')$ are called {\it equivariantly equivalent} if there exists an equivalence $f:G\longrightarrow G'$ which is in addition a Lie group isomorphism (this holds if and only if $\vp=df|_e^*\vp'$, where $df|_e:\ggo\longrightarrow\ggo'$ is the corresponding Lie algebra isomorphism).

We fix a $7$-dimensional vector space $\ggo$ and for each Lie bracket $\mu$ on $\ggo$, we consider $(G_\mu,\vp)$, the simply connected Lie group $G_\mu$ with Lie algebra $(\ggo,\mu)$ endowed with the left-invariant $G_2$-structure defined by the positive $3$-form on $\ggo$ given by
\begin{align}
\vp:=&e^{127}+e^{347}+e^{567}+e^{135}-e^{146}-e^{236}-e^{245} = \omega\wedge e^7+\rho^+, \label{phi-int}
\end{align}
where $\{ e_1,\dots,e_7\}$ is a basis of $\ggo$ (orthonormal with respect to the inner product $\ip_\vp$ induced by $\vp$).

\begin{theorem}\label{clasif}
Any Lie group endowed with a left-invariant ERP $G_2$-structure is equivalent up to scaling to $(G_\mu,\vp)$, where $\mu$ is exactly one of the following Lie brackets given in Table \ref{alg}:
\begin{equation}\label{clas-list}
\mu_B, \quad \mu_{M1}, \quad \mu_{M2}, \quad \mu_{M3}, \quad \mu_J.
\end{equation}
Moreover, in order to obtain a classification up to equivariant equivalence and scaling, exactly the structures
$(G_{ \mu_{rt}},\vp)$, $r,t\in\RR$, $(r,t)\ne(0,0)$, must be added to the list \eqref{clas-list} (see also Table \ref{alg}).  The structures $(G_{ \mu_{rt}},\vp)$ are all equivalent to $(G_{\mu_B},\vp)$ and the family of Lie algebras $\mu_{rt}$, $r,t\in\RR$ is pairwise non-isomorphic (note that $\mu_{00}=\mu_B$).
\end{theorem}

More friendly descriptions of these Lie algebras are given in Examples \ref{B}, \ref{rt}, \ref{M1}, \ref{M2}, \ref{M3} and \eqref{Jdef} (see \eqref{struc} first).  The proof of this theorem follows from Propositions \ref{n4-clasif}, \ref{n5-clasif} and \ref{n6-clasif} and is developed in Sections \ref{n4-sec}, \ref{n5-sec} and \ref{n6-sec}, after some preliminary material given in Section \ref{preli}.

We now list some interesting properties of the ERP $G_2$-structures obtained in the classification:

\begin{enumerate}[{\small $\bullet$} ]
\item $\mu_B$ was originally found in \cite{Bry}, $\mu_J$ in \cite{LS-ERP}, the $\mu_{Mi}$'s in \cite{ERP} and the curve from \cite{FinRff2} belongs to the $2$-parameter family $ \mu_{rt}$.

\item They are all steady Laplacian solitons and expanding Ricci solitons (see \cite[Corollary 4.8]{ERP}).

\item They all have torsion $2$-form equal to $\tau=e^{12}-e^{56}$.

\item With the only exception of $\mu_J$, they are all exact; indeed,
$$
\vp=d_\mu\left(3\tau - (\tr{A_1})^{-1} e^{34}\right), \qquad A_1:=\ad_\mu{e_7}|_{\spann\{ e_3,e_4\}}.
$$
Note that $(\tr{A_1})^{-1}=\tfrac{3}{2}$, $\tfrac{\sqrt{30}}{3}$, $3$, $\sqrt{6}$ for $\mu$ given by $\mu_B$ (or $ \mu_{rt}$), $\mu_{M1}$, $\mu_{M2}$, $\mu_{M3}$, respectively.  On the contrary, one obtains that $(G_{\mu_J},\vp)$ is not exact by using that the ERP condition implies that $e^{347}$ would be exact, which is impossible since $de^{ij}\perp e^{347}$ for any $i,j$.

\item $\hg:=\spann\{ e_1,\dots,e_6\}$ is always a unimodular ideal and the corresponding $\SU(3)$-structures $(\hg,\omega,\rho^+)$ are all {\it half-flat} (i.e.\ $d\omega^2=0$ and $d\rho^+=0$).  It is in addition {\it coupled} for $\mu_B$ (i.e.\ $d\omega=\tfrac{1}{3}\rho^+$).  For $\mu_J$, it is straightforward to see that the corresponding $\SU(3)$-structure on the ideal $\hg_1:=\{ e_1-e_2+e_7\}^\perp$ is {\it symplectic half-flat} (i.e.\ $d\omega_1=0$ and $d\rho_1^+=0$).

\item All Lie algebras in \eqref{clas-list} are completely solvable and the only unimodular one is $\mu_J$.  It was proved in \cite[Theorem 6.7]{FinRff2} that $G_{\mu_J}$ is the only unimodular Lie group admitting an ERP $G_2$-structure. The question of whether $G_{\mu_J}$ admits a lattice is still open.

\item For each $\mu$ in \eqref{clas-list}, the simply connected Lie group $G_\mu$ is the only Lie group with Lie algebra $\mu$; indeed, the center of $G_\mu$ is trivial since the center of the Lie algebra is so and the exponential function is a diffeomorphism.

\item The Betti numbers of each Lie algebra are given in Table \ref{bin}, together with some information on the respective nilradical $\ngo$.

\item Concerning symmetries, we have included in Table \ref{bin} for each of the examples both the subgroup of automorphisms of $\vp$ and the subgroup of isometries of $\ip_\vp$ which are also Lie group automorphisms (see Section \ref{sym-sec} for a more detailed study of symmetries).
\end{enumerate}

{\small \begin{table}

\renewcommand{\arraystretch}{1.6}
$$
\begin{array}{|c|l|} \hline
\mu_{B} &de^7=0, \quad de^3=\tfrac{1}{3}e^{37}, \quad de^4=\tfrac{1}{3}e^{47}, \quad de^1=-\tfrac{1}{6}e^{17}, \quad de^2=-\tfrac{1}{6}e^{27}, \\
& de^5=\tfrac{1}{3}e^{14}+\tfrac{1}{3}e^{23}+\tfrac{1}{6}e^{57}, \quad de^6=\tfrac{1}{3}e^{13}-\tfrac{1}{3}e^{24}+\tfrac{1}{6}e^{67}.\\
\hline
 & de^7=0, \quad de^3=\tfrac{\sqrt{30}}{30}e^{37}, \quad de^4=\tfrac{\sqrt{30}}{15}e^{47},\\
& de^1=-\tfrac{\sqrt{5}}{30}e^{14}-\tfrac{10+\sqrt{30}}{60}e^{17}-\tfrac{\sqrt{5}}{30}e^{23}-\tfrac{5-\sqrt{30}}{30}e^{36}-\tfrac{5-\sqrt{30}}{30}e^{45}-\tfrac{\sqrt{5}}{30}e^{57},\\
\mu_{M1}&de^2=\tfrac{\sqrt{5}}{6}e^{13}+\tfrac{\sqrt{5}}{30}e^{24}-\tfrac{10-\sqrt{30}}{60}e^{27}-\tfrac{1}{6}e^{35}+\tfrac{5-\sqrt{30}}{30}e^{46}-\tfrac{\sqrt{5}}{30}e^{67},\\
&de^5=\tfrac{5+\sqrt{30}}{30}e^{14}-\tfrac{\sqrt{5}}{30}e^{17}+\tfrac{5+\sqrt{30}}{30}e^{23}-\tfrac{\sqrt{5}}{30}e^{36}-\tfrac{\sqrt{5}}{30}e^{45}+\tfrac{10-\sqrt{30}}{60}e^{57},\\
&de^6=
\tfrac{1}{6}e^{13}-\tfrac{5+\sqrt{30}}{30}e^{24}-\tfrac{\sqrt{5}}{30}e^{27}+\tfrac{\sqrt{5}}{6}e^{35}+\tfrac{\sqrt{5}}{30}e^{46}+\tfrac{10+\sqrt{30}}{60}e^{67}.\\
\hline
&de^7=de^3=0, \quad  de^4=\tfrac{1}{3}e^{47}, \\
\mu_{M2}  &de^1=-\tfrac{1}{6}e^{13}-\tfrac{1}{3}e^{17}, \quad  de^2=-\tfrac{1}{3}e^{14}+\tfrac{1}{6}e^{23}-\tfrac{1}{3}e^{35},\\
&de^5=\tfrac{1}{3}e^{14}+\tfrac{1}{3}e^{23}-\tfrac{1}{6}e^{35}, \quad de^6=-\tfrac{1}{3}e^{24}+\tfrac{1}{6}e^{36}-\tfrac{1}{3}e^{45}+\tfrac{1}{3}e^{67}. \\
\hline
&de^7=de^3=0,  \quad de^4=\tfrac{\sqrt{6}}{6}e^{47},\\
&de^1=-\tfrac{\sqrt{2}}{12}e^{14}-\tfrac{1}{6}^{17}+\tfrac{\sqrt{2}}{6}e^{23}-\tfrac{1}{6}e^{36}-\tfrac{2-\sqrt{6}}{12}e^{45}-\tfrac{\sqrt{2}}{12}e^{57},\\
\mu_{M3} &de^2=\tfrac{\sqrt{2}}{6}e^{13}+\tfrac{\sqrt{2}}{12}^{24}-\tfrac{1}{6}e^{27}-\tfrac{1}{6}e^{35}+\tfrac{2-\sqrt{6}}{12}e^{46}-\tfrac{\sqrt{2}}{12}e^{67},\\
&de^5=\tfrac{2+\sqrt{6}}{12}e^{14}-\tfrac{\sqrt{2}}{12}^{17}+\tfrac{1}{6}e^{23}+\tfrac{\sqrt{2}}{6}e^{36}-\tfrac{\sqrt{2}}{12}e^{45}+\tfrac{1}{6}e^{57},\\
& de^6=\tfrac{1}{6}e^{13}-\frac{2+\sqrt{6}}{12}e^{24}-\tfrac{\sqrt{2}}{12}e^{27}+\tfrac{\sqrt{2}}{6}e^{35}+\tfrac{\sqrt{2}}{12}e^{46}+\tfrac{1}{6}e^{67}.\\
\hline
\mu_{J} & de^7=de^3=de^4=0,  \quad  de^1=\tfrac{\sqrt{2}}{6}e^{14}-\tfrac{1}{6}e^{17}-\tfrac{\sqrt{2}}{6}e^{23}-\tfrac{1}{3}e^{36},\\
&de^2=-\tfrac{\sqrt{2}}{6}e^{13}-\tfrac{\sqrt{2}}{6}e^{24}-\tfrac{1}{6}e^{27}+\tfrac{1}{3}e^{46}, \quad de^5=\tfrac{1}{2}e^{57}, \quad
de^6=\tfrac{1}{3}e^{13}-\tfrac{1}{3}e^{24}-\tfrac{1}{6}e^{67}.\\ \hline
 &
de^7=0, \quad de^3=\tfrac{\sqrt{30}}{30}e^{37}-\tfrac{1}{3}re^{47}, \quad de^4=\tfrac{1}{3}e^{47}+\tfrac{1}{3}re^{37}, \\
\mu_{rt} & de^1=-\tfrac{1}{6}e^{17}-\tfrac{1}{3}te^{27}, \quad de^2=-\tfrac{1}{6}e^{27}+\tfrac{1}{3}te^{17}, \\
& de^5=\tfrac{1}{3}e^{14}+\tfrac{1}{3}e^{23}+\tfrac{1}{6}e^{57}+\tfrac{1}{3}(r+t)e^{67}, \quad de^6=\tfrac{1}{3}e^{13}-\tfrac{1}{3}e^{24}+\tfrac{1}{6}e^{67}-\tfrac{1}{3}(r+t)e^{57}.\\ \hline
\end{array}
$$
\caption{Structure coefficients.}\label{alg}
\end{table}}

{\small \begin{table}
\renewcommand{\arraystretch}{1.6}
$$
\begin{array}{|c||c|c|c|c|c|c|c|c|c|c|}\hline
 & b_1 & b_2 & b_3 & b_4 & b_5 & b_6 & \dim{\ngo} & \mbox{nilp. deg.} & \Aut(\mu)\cap G_2 & \Aut(\mu)\cap\Or(7) \\
\hline\hline
\mu_{B}   & 1 & 2 & 2 & 2 & 2 & 0 & 6 & \mbox{2-step} & S^1\times S^1 &  \ZZ_2\ltimes (S^1\times S^1)\\
\hline
\mu_{M1} & 1 & 0 & 0 & 1 & 1 & 0 & 6 & \mbox{4-step} & \ZZ_2 &  \ZZ_2\times \ZZ_2 \\
\hline
\mu_{M2} & 2 & 1 & 0 & 1 & 2 & 1 & 5 & \mbox{3-step} & \ZZ_2 &  \ZZ_2\times \ZZ_2\times \ZZ_2\\
\hline
\mu_{M3} & 2 & 2 & 2 & 2 & 2 & 1 & 5 & \mbox{2-step} & \ZZ_4 & D_4\times  \ZZ_2\\
\hline
\mu_{J}   & 3 & 3 & 1 & 1 & 3 & 3 & 4 & \mbox{abelian} & \Sl_2(\ZZ_3) & S_4\ltimes\ZZ_2^4\\
\hline
 \mu_{rt}   & 1 & 2 & 2 &2 & 2 & 0 & 6 & \mbox{2-step} & S^1\times S^1 & S^1\times S^1\\
\hline
\end{array}
$$
\caption{Betti numbers, nilradical $\ngo$ and symmetries.}\label{bin}
\end{table}}

\section{Preliminaries}\label{preli}

A left-invariant $G_2$-structure on a Lie group is determined by a positive $3$-form on the Lie algebra $\ggo$, which  will always be given by
\begin{align}
\vp:=&e^{127}+e^{347}+e^{567}+e^{135}-e^{146}-e^{236}-e^{245} \notag \\
=&\omega_7\wedge e^7+\omega_3\wedge e^3+\omega_4\wedge e^4+e^{347}, \label{phi}
\end{align}
where $\{ e_1,\dots,e_7\}$ is an orthonormal basis of $\ggo$ and
\begin{equation}\label{omeg}
\omega_7:=e^{12}+e^{56}, \quad \omega_3:=e^{26}-e^{15}, \quad \omega_4:=e^{16}+e^{25}.
\end{equation}

The following is the main structure result in \cite{ERP}.  Let $\theta$ denote the usual representation of $\glg_4(\RR)$ on $\Lambda^2(\RR^4)^*$, that is, $\theta(E)\alpha=-\alpha(E\cdot,\cdot)-\alpha(\cdot,E\cdot)$ for all $E\in\glg_4(\RR)$ and $\alpha\in\Lambda^2(\RR^4)^*$.

\begin{theorem}\label{main}\cite[Theorem 4.7 and Proposition 4.9]{ERP}
Every Lie group endowed with a left-invariant ERP $G_2$-structure is equivariantly equivalent (up to scaling) to some $(G,\vp)$ with torsion
$\tau=e^{12}-e^{56}$, such that $\vp$ is as in \eqref{phi} and the following conditions hold for the Lie algebra $\ggo$ of $G$:
\begin{itemize}
\item[(i)] $\hg:=\spann\{ e_1,\dots,e_6\}$ is a unimodular ideal, $\ggo_0:=\spann\{ e_7,e_3,e_4\}$ is a Lie subalgebra, $\ggo_1:=\spann\{ e_1,e_2,e_5,e_6\}$ is an abelian ideal and $\hg_1:=\spann\{ e_3,e_4\}$ is an abelian subalgebra.

\item[(ii)] $\theta(\ad{e_7}|_{\ggo_1})\tau=\frac{1}{3}\omega_7$, $\theta(\ad{e_3}|_{\ggo_1})\tau=\frac{1}{3}\omega_3$ and $\theta(\ad{e_4}|_{\ggo_1})\tau=\frac{1}{3}\omega_4$.

\item[(iii)] $\theta(\ad{e_7}|_{\ggo_1})\omega_7+\theta(\ad{e_3}|_{\ggo_1})\omega_3+\theta(\ad{e_4}|_{\ggo_1})\omega_4 = \tau +(\tr{\ad{e_7}|_{\hg_1}})\omega_7$.
\end{itemize}
Conversely, if $\ggo$ satisfies conditions (i)-(iii), then $(G,\vp)$ is an ERP $G_2$-structure with torsion $\tau=e^{12}-e^{56}$.
\end{theorem}

Structurally, it follows from the above theorem that, up to equivariant equivalence, the Lie bracket $\mu$ of the Lie algebra $\ggo$ of any ERP $(G,\vp)$ with $\tau=e^{12}-e^{56}$ is given by,
\begin{equation}\label{struc}
\mu=(A_1,A,B,C),
\end{equation}
in the sense that $\mu$ is determined by the $2\times 2$ matrix $A_1:=\ad_\mu{e_7}|_{\hg_1}$ and the three $4\times 4$ traceless matrices $A:=\ad_\mu{e_7}|_{\ggo_1}$, $B=\ad_\mu{e_3}|_{\ggo_1}$, $C:=\ad_\mu{e_4}|_{\ggo_1}$, which must satisfy conditions (ii) and (iii). The Jacobi condition, on the other hand, is equivalent to
$$
[A,B]=aB+cC, \quad [A,C]=bB+dC, \quad [B,C]=0, \qquad
A_1=\left[\begin{smallmatrix} a&b\\ c&d\end{smallmatrix}\right].
$$
It was obtained in \cite[Section 5]{ERP} that there are only three possibilities for the underlying vector space $\ngo$ of the nilradical of $\mu$ and that the following additional conditions must hold in each case (up to equivariant equivalence):

\begin{enumerate}[{\small $\bullet$} ]
\item $\ngo=\ggo_1$ ($\dim{\ngo}=4$): this is equivalent to $\mu$ unimodular and one has that $A_1=0$, the matrices $A,B,C$ are all symmetric, they pairwise commute and $\left\{\sqrt{3}A, \sqrt{3}B, \sqrt{3}C\right\}$ is orthonormal relative to the usual inner product $\tr{XY^t}$.  In particular, only one Lie algebra (up to isomorphism) shows up in this case.

\item $\ngo=\RR e_4\oplus\ggo_1$ ($\dim{\ngo}=5$): $A,B$ are symmetric, $C$ is nilpotent, $a=b=c=0$ and $d>0$.

\item $\ngo=\hg$ ($\dim{\ngo}=6$): $A_1$ and $A$ are normal, $B$ and $C$ are nilpotent and
\begin{enumerate}[{\small $\bullet$} ]
\item either $A_1=\left[\begin{smallmatrix} a&0\\ 0&d\end{smallmatrix}\right]$, with $a\leq d$, $a+d>0$,
\item or $A_1=\left[\begin{smallmatrix} a&b\\ -b&a\end{smallmatrix}\right]$, with $a>0$, $b\ne 0$.
\end{enumerate}
\end{enumerate}

Up to equivalence, only five examples of left-invariant ERP $G_2$-structures on a Lie group were known, one with $\dim{\ngo}=4$ and two in each of the other two cases (see \cite[Section 5]{ERP}).

We recall from \cite[Proposition 4.4]{ERP} that in the non-unimodular case (i.e.\ $\dim{\ngo}=5,6$), the equivariant equivalence among the set of closed $G_2$-structures $(G_\mu,\vp)$, where $\mu=(A_1,A,B,C)$ is as in \eqref{struc}, is determined by the action of the subgroup $U_{\hg,\tau}=U_0\cup U_0g\subset G_2$ (see \cite[Lemma 2.3]{ERP}), where $G_2:=\{ h\in\Gl_7(\RR)):h^*\vp=\vp\}\subset\SO(7)$,
\begin{equation}\label{defU0}
U_0:=\left\{\left[\begin{smallmatrix}
1&&& \\ &h_1&& \\ &&h_2&\\ &&&h_3
\end{smallmatrix}\right] : h_i\in\SO(2), \; h_1h_2h_3=I\right\},
\end{equation}
and
$$
ge_7=-e_7, \qquad g_1:=g|_{\hg_1}=
\left[\begin{smallmatrix}
1& \\ &-1
\end{smallmatrix}\right], \qquad g_2:=g|_{\ggo_1}=
\left[\begin{smallmatrix}
0&0&1&0\\ 0&0&0&-1\\ -1&0&0&0\\ 0&1&0&0
\end{smallmatrix}\right].
$$
The action is given as follows (see \cite[(30)]{ERP}): if $h\in U_0$, say with $h_1=
\left[\begin{smallmatrix}
x&y \\ -y&x
\end{smallmatrix}\right]$, $x^2+y^2=1$ and $h_4:=
\left[\begin{array}{cc}
h_2&0 \\ 0&h_3
\end{array}\right]$, $h_2,h_3\in\SO(2)$, then
\begin{equation}\label{equiv1}
h\cdot\mu = \left(h_1A_1h_1^{-1}, h_4Ah_4^{-1},  h_4(xB-yC)h_4^{-1}, h_4(yB+xC)h_4^{-1}\right),
\end{equation}
and
\begin{equation}\label{equiv2}
g\cdot\mu = \left(-g_1A_1g_1^{-1}, -g_2Ag_2^{-1}, g_2Bg_2^{-1}, -g_2Cg_2^{-1}\right).
\end{equation}

\section{Case $\dim{\ngo}=4$}\label{n4-sec}

In this section, we obtain a classification of left-invariant ERP $G_2$-structures on Lie groups with nilradical of dimension $4$, up to equivariant equivalence and scaling.  Recall from Section \ref{preli} that there is a unique Lie group $G$ involved in this case.  Since a $G_2$-structure $\vp$ is ERP if and only if $-\vp$ is so, it is enough to consider $G_2$-structures with a given orientation.

As explained in \cite[Section 6]{ERP}, each $\mu$ in the algebraic subset $\lca\subset\Lambda^2\ggo^*\otimes\ggo$ of all Lie brackets on $\ggo$ is identified with $(G_\mu,\vp)$, where $G_\mu$ denotes the simply connected Lie group with Lie algebra $(\ggo,\mu)$ and $\vp$ is defined by \eqref{phi}.  The orbit $\Gl_7^+(\RR)\cdot\mu$ therefore parametrizes the set of all left-invariant $G_2$-structures on $G_\mu$ with the same orientation as $\vp$, due to the equivariant equivalence,
$$
(G_{h\cdot\mu},\vp) \simeq (G_\mu,\vp(h\cdot,h\cdot,h\cdot)), \qquad\forall h\in\Gl_7(\RR).
$$
Note that two elements in $\lca$ are equivariantly equivalent if and only if they belong to the same $G_2$-orbit, and that they are in the same
$\Or(7)$-orbit if and only if they are equivariantly isometric as Riemannian metrics.  It is known that both assertions hold, without the word
`equivariantly', for completely real solvable Lie brackets (see \cite{Alk}).  In the light of Theorem \ref{main}, any $(G_\mu,\vp)$ will be assumed from
now on to have the structure $\mu=(A_1,A,B,C)$ as in \eqref{struc}.

Assume that $(G_\mu,\vp)$ and $(G_{h\cdot\mu},\vp)$ are both ERP with $\tau_\mu=\tau_{h\cdot\mu}=e^{12}-e^{56}$ for some $h\in\Gl_7(\RR)$, $\det{h}>0$.
Since both $(G_\mu,\ip)$ and $(G_{h\cdot\mu},\ip)$ are solvsolitons by \cite[Corollary 4.8]{ERP}, it follows from the uniqueness of solvsolitons (up to
equivariant isometry and scaling) on a given Lie group (see \cite{solvsolitons} or \cite{BhmLfn}) that we can assume  $h\in\SO(7)$.  Thus
$\Ricci_{h\cdot\mu}=h\Ricci_\mu h^{-1}$ and so by \cite[Proposition 3.9, (iv)]{ERP},
\begin{equation}\label{h}
h=\left[\begin{smallmatrix}h_1 &\\ & h_2\end{smallmatrix}\right], \qquad h_1\in\Or(3), \quad h_2\in\Or(4), \quad\det{h_1}=\det{h_2},
\end{equation}
with respect to the decomposition $\ggo=\ggo_0\oplus\ggo_1$.

Relative to the matrix $J$ of $\tau$ with respect to the basis $\{ e_1,e_2,e_5,e_6\}$, given by,
$$
J:=\left[\begin{smallmatrix}
 0&-1&& \\ 1&0&& \\  &&0&1\\ &&-1&0
\end{smallmatrix}\right],
$$
one obtains the following classic orthogonal decompositions,
$$
\slg_4(\RR) = \sog(4) \oplus \sym_0(4), \quad \sym_0(4) = \pg_1\oplus\pg_2, \quad \spg(2,\RR)=\ug(2)\oplus\pg_1, \quad
 \glg_2(\CC)=\ug(2)\oplus\pg_2,
$$
where
\begin{equation}\label{p1}
\pg_1:=\left\{\left[\begin{smallmatrix}
 a&b&e&f \\ b&-a&-f&e \\  e&-f&c&d\\ f&e&d&-c
\end{smallmatrix}\right] : a,\dots,f\in\RR\right\},
\end{equation}
and $\pg_2$ has the following orthogonal basis,
\begin{equation}\label{Ti}
T_7:=\tfrac{1}{6}\left[\begin{smallmatrix}
 -1&&& \\ &-1&& \\  &&1&\\ &&&1
\end{smallmatrix}\right], \quad
T_3:=\tfrac{1}{6}\left[\begin{smallmatrix} &&&1 \\ &&1& \\  &1&&\\ 1&&&
\end{smallmatrix}\right], \quad
T_4:=\tfrac{1}{6}\left[\begin{smallmatrix} &&1&0 \\ &&0&-1 \\  1&0&&\\ 0&-1&&
\end{smallmatrix}\right].
\end{equation}
Note that $|T_i|^2=\tfrac{1}{9}$ for all $i$ and
\begin{equation}\label{titaT}
\theta(T_7)\tau=\tfrac{1}{3}\omega_7, \qquad \theta(T_3)\tau=\tfrac{1}{3}\omega_3, \qquad \theta(T_4)\tau=\tfrac{1}{3}\omega_4.
\end{equation}

According to the structural results given in Section \ref{preli}, each ERP $G_2$-structure with $\dim{\ngo}=4$ is therefore determined by an abelian
subalgebra $\ag\subset\sym_0(4)$ endowed with an orthogonal basis $\{ A,B,C\}$ such that $|A|^2=|B|^2=|C|^2=\frac{1}{3}$ and
$$
A=E+T_7, \quad B=F+T_3, \quad C=G+T_4,
$$
for some uniquely determined $E,F,G\in\pg_1$.  Thus $\{ E,F,G\}$ is orthogonal and $|E|^2=|F|^2=|G|^2=\frac{2}{9}$.  The only known example in this case is $\mu_J:=(0,A,B,C)$ (see \cite[Example 5.4]{ERP} and \cite[Example 4.7]{LS-ERP}), where
\begin{equation}\label{Jdef}
A=\tfrac{1}{6}\left[\begin{smallmatrix} -1&&& \\ &-1&& \\ &&3&\\ &&&-1
\end{smallmatrix}\right], \quad
B=\tfrac{1}{6}\left[\begin{smallmatrix} 0&-\sqrt{2}&0&2 \\ -\sqrt{2}&0&0&0\\ 0&0&0&0\\ 2&0&0&0
\end{smallmatrix}\right], \quad
C=\tfrac{1}{6}\left[\begin{smallmatrix} \sqrt{2}&0&0&0 \\  0&-\sqrt{2}&0&-2 \\ 0&0&0&0 \\ 0&-2&0&0
\end{smallmatrix}\right],
\end{equation}
and so
$$
E=\tfrac{1}{6}\left[\begin{smallmatrix} 0&&& \\ &0&& \\ &&2&\\ &&&-2
\end{smallmatrix}\right], \quad
F=\tfrac{1}{6}\left[\begin{smallmatrix} 0&-\sqrt{2}&0&1 \\ -\sqrt{2}&0&-1&0\\ 0&-1&0&0\\ 1&0&0&0
\end{smallmatrix}\right], \quad
G=\tfrac{1}{6}\left[\begin{smallmatrix} \sqrt{2}&0&-1&0 \\  0&-\sqrt{2}&0&-1 \\ -1&0&0&0 \\ 0&-1&0&0
\end{smallmatrix}\right].
$$
In particular, $G_{\mu_J}$ is the only unimodular Lie group admitting an ERP $G_2$-structure by \cite[Theorem 6.7]{FinRff2} (see also \cite[Proposition 5.2]{ERP}).  In the following proposition, we show in addition that $G_{\mu_J}$ admits exactly one ERP $G_2$-structure up to equivariant equivalence and scaling (cf.\ \cite[Remark 5.3]{ERP}).

\begin{proposition}\label{n4-clasif}
$(G_{\mu_J},\vp)$ is the unique ERP $G_2$-structure up to equivariant equivalence and scaling among the class of unimodular Lie groups (or equivalently, on Lie groups with nilradical of dimension $4$) endowed with a $G_2$-structure.
\end{proposition}

\begin{proof}
Let $\ag$ denote the abelian subalgebra associated to $\mu_J$ and assume that $\overline{\ag}\subset\sym_0(4)$ is another ERP abelian subalgebra with corresponding basis
$$
\{\overline{A}=\overline{E}+T_7, \overline{B}=\overline{F}+T_3, \overline{C}=\overline{G}+T_4\}.
$$
It is well known that there exists $h_2\in\SO(4)$ such that $h_2\ag h_2^{-1}=\overline{\ag}$.  If $h_1\in\SO(3)$ is defined in terms of the above basis by
$h_2Ah_2^{-1}=h_1\overline{A}$ and so on, then $\overline{\mu}=h\cdot\mu_J$, where $h$ is as in \eqref{h}.  Thus $h_1=I$ can be assumed up to equivariance equivalence, since there is an $u\in U_{\ggo_1,\tau}$ (see \cite[(10)]{ERP}) such that $u|_{\ggo_0}=h_1^{-1}$; hence
\begin{equation}\label{hc}
h_2(E+T_7)h_2^{-1}=\overline{E}+T_7, \quad h_2(F+T_7)h_2^{-1}=\overline{F}+T_3, \quad h_2(G+T_7)h_2^{-1}= \overline{G}+T_4.
\end{equation}
It follows from \eqref{p1} and \eqref{Ti} that
$$
\la\overline{A}e_1,e_1\ra = a-\tfrac{1}{6}, \qquad \la\overline{A}e_2,e_2\ra = -a-\tfrac{1}{6}, \qquad\mbox{for some}\; a\in\RR,
$$
but since $\Spec(\overline{A})=\{ -\frac{1}{6},\unm\}$, we obtain that $a=0$ and $e_1$, $e_2$ are both eigenvectors of $\overline{A}$ with eigenvalue $-\frac{1}{6}$.  Thus $\overline{E}$ has the form
$$
\overline{E}=\tfrac{1}{6}\left[\begin{smallmatrix} 0&0&& \\ 0&0&& \\ &&c&d\\ &&d&-c
\end{smallmatrix}\right], \qquad c,d\in\RR, \qquad c^2+d^2=\tfrac{1}{9},
$$
and so there exists
$$
u=\left[\begin{smallmatrix} I&\\ &u_2 \end{smallmatrix}\right]\in U_{\ggo_1,\tau}, \qquad u_2:=\left[\begin{smallmatrix} u_3 &\\ & u_3^{-1}\end{smallmatrix}\right]\in U(2), \qquad u_3\in\SO(2),
$$
such that $u_2\overline{E}u_2^{-1}=E$ (see also \eqref{equiv1}).  Note that $u_2T_iu_2^{-1}=T_i$ for all $i$.  This allows us to assume that $\overline{E}=E$, up to equivaraint equivalence.  It now follows from \eqref{hc} that $h_2$ commutes with $A$ and so  $h_2e_5=\pm e_5$, which implies that
$$
\overline{F}e_5=-T_3e_5=-\tfrac{1}{6}e_2 , \qquad \overline{G}e_5=-T_4e_5=-\tfrac{1}{6}e_1.
$$
By \eqref{p1}, the matrices $\overline{F}$ and $\overline{G}$ considerably simplify as follows,
$$
\overline{F}=\tfrac{1}{6}\left[\begin{smallmatrix}a&b&0&1\\b&-a&-1&0\\0&-1&0&0\\1&0&0&0\end{smallmatrix}\right], \qquad
\overline{G}=\tfrac{1}{6}\left[\begin{smallmatrix}a'&b'&-1&0\\b'&-a'&0&-1\\-1&0&0&0\\0&-1&0&0\end{smallmatrix}\right],
$$
 and condition
 $[\overline{B},\overline{C}]=0$ gives that
$$
\overline{B}=\tfrac{1}{6}\left[\begin{smallmatrix} a&b&0&2 \\ b&-a&0&0\\ 0&0&0&0\\ 2&0&0&0
\end{smallmatrix}\right], \qquad
\overline{C}=\tfrac{1}{6}\left[\begin{smallmatrix} -b&a&0&0 \\  a&b&0&-2 \\ 0&0&0&0 \\ 0&-2&0&0
\end{smallmatrix}\right], \qquad a^2+b^2=2.
$$
Notice that $\mu_J$ corresponds to $a=0$, $b=-\sqrt{2}$.  Finally, recall from \eqref{hc} that $\overline{B}$ and $\overline{C}$ are respectively conjugate
to $B$ and $C$, so $\tr{\overline{B}^3}= \frac{1}{18} a$ can not depend on $a$, so $a=0$, and since if $a=0$ then $\tr{\overline{C}^3}= \pm
\frac{\sqrt{2}}{18} $ depending on whether $b=\pm\sqrt{2}$, we obtain that $\overline{\mu}=\mu_J$, concluding the proof.
\end{proof}

\section{Case $\dim{\ngo}=5$}\label{n5-sec}

We classify in this section, up to equivariant equivalence and scaling, all left-invariant ERP $G_2$-structures on Lie groups with nilradical of dimension $5$.  There are only two known examples in this case, which we next describe.

From now on, all matrices in $\glg_6(\RR)$ will be written in terms of the orthogonal basis of $\Lambda^2\ggo_1^*$ defined by
\begin{equation}\label{beta}
\bca:=\{\tau,\overline{\omega}_3,\overline{\omega}_4,\omega_7,\omega_3,\omega_4\}, \quad\mbox{where} \quad\overline{\omega}_3:=e^{26}+e^{15},\quad \overline{\omega}_4:=e^{16}-e^{25},
\end{equation}
see \eqref{omeg} for the definition of the $\omega_i$'s.   Note that each element in $\bca$ has norm equal to $\sqrt{2}$.

\begin{example}\label{M2}\cite[Example 5.5]{ERP}
Let $(G_{\mu_{M2}},\vp)$ be the ERP $G_2$-structure with Lie bracket $\mu_{M2}$ given by
$$
(A_1)_{M2}=\tfrac{1}{3}\left[\begin{smallmatrix} 0&\\ &1\end{smallmatrix}\right], \quad A_{M2}=\tfrac{1}{3}\left[\begin{smallmatrix}-1 &  &  &  \\  & 0 & &  \\
& & 0&\\  &  & & 1\\ \end{smallmatrix}\right], \quad B_{M2}=\tfrac{1}{6}\left[\begin{smallmatrix}-1 &  &  &  \\  & 1 & 2 &  \\  & 2 & 1 &  \\  &  &  & 1
\\\end{smallmatrix}\right], \quad
 C_{M2}=\tfrac{1}{3}\left[\begin{smallmatrix}0 & & &  \\-1 & 0 &  &  \\1 & 0 & 0 &  \\0  & -1 & 1 & 0 \\\end{smallmatrix}\right].
$$
It follows by an easy computation that
$$
\theta(A_{M2})=\tfrac{1}{3}\left[\begin{smallmatrix} & & &1&0&0\\ & & &0&-1&0\\ & & &0&0&0\\1&0&0& & & \\0&-1&0&&&\\0&0&0&&&
\end{smallmatrix}\right], \theta(B_{M2})=\tfrac{1}{3}\left[\begin{smallmatrix}
&&&0&1&0\\&&&-1&0&0\\&&&0&0&1\\0&-1&0&&&\\1&0&0&&&\\0&0&1&&&
\end{smallmatrix}\right],\theta(C_{M2})=\tfrac{1}{3}\left[\begin{smallmatrix} &&&0&0&1\\&&-1&0&0&0\\&1&&0&1&0\\0&0&0&&&1\\0&0&1&&0&\\1&0&0&-1&&
\end{smallmatrix}\right],
$$
with respect to the basis $\bca$ of $\Lambda^2\ggo_1^*$.
\end{example}

\begin{example}\label{M3}\cite[Example 5.8]{ERP}
The ERP $G_2$-structure $(G_{\mu_{M3}},\vp)$ has Lie bracket $\mu_{M3}$ given by
$$
(A_1)_{M3}=\tfrac{1}{6}\left[\begin{smallmatrix} 0&0\\0 & \sqrt{6} \end{smallmatrix}\right], \quad A_{M3}=\tfrac{1}{12}\left[\begin{smallmatrix} -2 & 0 &
-\sqrt{2} & 0\\0 & -2 & 0 & -\sqrt{2} \\-\sqrt{2} & 0 & 2 & 0 \\0 & -\sqrt{2} & 0 & 2\\\end{smallmatrix}\right],
$$
$$
B_{M3}=\tfrac{1}{6}\left[\begin{smallmatrix}0 & \sqrt{2} & 0 & 1 \\\sqrt{2} & 0 & 1 & 0 \\0 & 1 & 0 & -\sqrt{2} \\1 & 0 & -\sqrt{2} & 0\\
\end{smallmatrix}\right], \quad
C_{M3}=\tfrac{1}{12}\left[\begin{smallmatrix}-\sqrt{2} & 0 & 2-\sqrt{6} & 0 \\  0 & \sqrt{2} & 0 & -2+\sqrt{6}\\2+\sqrt{6} & 0 & \sqrt{2} & 0 \\0 &
-2-\sqrt{6} & 0 &-\sqrt{2}\\\end{smallmatrix}\right].
$$
It is easy to check that $\theta(A_{M3})$, $\theta(B_{M3})$ and $\theta(C_{M3})$ are respectively equal to
$$ \tfrac{1}{6}\left[\begin{smallmatrix}
&&&2&0&0\\&&&0&0&0\\&&&\sqrt{2}&0&0\\2&0&\sqrt{2}&&&\\0&0&0&&&\\0&0&0&&&
\end{smallmatrix}\right],  \quad \tfrac{1}{3}\left[\begin{smallmatrix}
&&&0&1&0\\&&&0&0&0\\&&&0&-\sqrt{2}&0\\0&0&0&&&\\1&0&-\sqrt{2}&&&\\0&0&0&&&
\end{smallmatrix}\right],   \quad \tfrac{1}{6}\left[\begin{smallmatrix}
&&&0&0&2\\&&&0&0&0\\&&&0&0&\sqrt{2}\\0&0&0&&&\sqrt{6}\\0&0&0&&0&\\2&0&\sqrt{2}&-\sqrt{6}&&
\end{smallmatrix}\right].
$$
\end{example}

Our aim in this section is to show that, up to equivariant equivalence and scaling, the above two examples are the only ERP $G_2$-structures on Lie groups with nilradical of dimension $5$.

We consider the operator $\ast_{\ggo_2}:\Lambda^2\ggo_2^*\longrightarrow\Lambda^2\ggo_2^*$, whose matrix in terms of $\bca$ is given by
$$
[\ast_{\ggo_2}]=
\left[\begin{smallmatrix}
-I & \\
& I
\end{smallmatrix}\right]\in\glg_6(\RR).
$$
If $M\in\slg_4(\RR)$, then $\theta(M)\ast_{\ggo_2}=-\theta(M^t)\ast_{\ggo_2}$ (see e.g.\ \cite[Lemma 2.1, (vi)]{ERP})
and since $\theta(M)^t=\theta(M^t)$, what this is asserting is that $\theta(M)\in\sog(3,3)$ for any $M\in\slg_4(\RR)$, that is,
\begin{equation}\label{magic}
\theta(M)=\left[\begin{matrix}M_1 &M_2\\M_2^t&M_3\end{matrix}\right], \qquad M_1^t=-M_1, \quad M_3^t=-M_3,
\end{equation}
for some $M_1,M_2,M_3\in\glg_3(\RR)$.  Note that $\theta:\slg_4(\RR)\rightarrow\sog(3,3)$ is clearly an isomorphism.

\begin{proposition}\label{n5-clasif}
Every Lie group endowed with a left-invariant ERP $G_2$-structure and having $5$-dimensional nilradical is equivariantly equivalent (up to scaling) to either $(G_{\mu_{M2}},\vp)$ or $(G_{\mu_{M3}},\vp)$.
\end{proposition}

\begin{proof}
We can assume that the Lie bracket $\mu$ of an ERP $G_2$-structure $(G_\mu,\vp)$ is given by $\mu=(A_1,A,B,C)$, where $A,B$ are symmetric, $C$ is nilpotent and $A_1=\left[\begin{smallmatrix} 0
& \\ & \delta\end{smallmatrix}\right]$, $\delta>0$ (see Section \ref{preli}).  Since $\tr{A}=\tr{B}=\tr{C}=0$, the matrices $\theta(A),\theta(B),\theta(C)$ have the form as in \eqref{magic} with respect to the basis $\bca$ given in \eqref{beta}.

It follows from Theorem \ref{main} (ii) and (iii) that $3\theta(A)$, $3\theta(B)$ and $3\theta(C)$ are respectively given by
$$
\left[\begin{smallmatrix}
&&&1&0&0\\
&&&a_{24}&a_{25}&a_{26}\\
&&&a_{34}&a_{35}&a_{36}\\
1&a_{24}&a_{34}&&&\\
0&a_{25}&a_{35}&&&\\
0&a_{26}&a_{36}&&&\\
\end{smallmatrix}\right],
\left[\begin{smallmatrix}
&&&0&1&0\\
&&&b_{24}&b_{25}&b_{26}\\
&&&b_{34}&b_{35}&b_{36}\\
0&b_{24}&b_{34}&&&\\
1&b_{25}&b_{35}&&&\\
0&b_{26}&b_{36}&&&\\
\end{smallmatrix}\right],
\left[\begin{smallmatrix}
0&0&0&0&0&1\\
0&0&c_{23}&c_{24}&c_{25}&-(a_{24}+b_{25})\\
0&-c_{23}&0&c_{34}&c_{35}&-(a_{34}+b_{35})\\
0&c_{24}&c_{34}&0&c_{45}&3\delta\\
0&c_{25}&c_{35}&-c_{45}&0&0\\
1&-(a_{24}+b_{25})&-(a_{34}+b_{35})&-3\delta&0&0\\
\end{smallmatrix}\right],
$$
for some $a_{ij},b_{ij},c_{ij}\in \RR$. Now, since $A_1,A,B,C$ satisfy Jacobi and $\theta$ is a representation, we know that if
$$
R:=[\theta(A),\theta(C)]-\delta \theta(C), \quad S:=[\theta(A),\theta(B)],\quad T:=[\theta(B),\theta(C)],
$$
then $R=S=T=0$.  From the first columns of $R,S$ and $T$ we easily obtain that
$$
c_{25} = b_{26}, c_{35} = b_{36}, b_{24} = a_{25}, b_{34} = a_{35}, c_{24} = a_{26}, c_{34} =
a_{36}, c_{45} = 0.
$$
We list below the other null expressions provided by $R=S=T=0$ that will be needed next:
\begin{align}
&T(4..6,2..3)=\tfrac{1}{9}\left[\begin{smallmatrix} -3 b_{26} \delta  -a_{35} c_{23} &-3 b_{36} \delta +a_{25} c_{23}\\-b_{35} c_{23}& b_{25} c_{23}\\3 a_{25} \delta -b_{36}
c_{23}&3 a_{35}\delta +b_{26} c_{23}
\end{smallmatrix}\right]=0\label{eq1}\\
&R(4..6,2..3)=\tfrac{1}{9}\left[\begin{smallmatrix}-6 a_{26} \delta -a_{34} c_{23}&-6 a_{36} \delta +a_{24}
c_{23}\\-3b_{26}\delta-a_{35}c_{23}&-3b_{36}\delta+a_{25} c_{23}\\3(2a_{24}+b_{25})\delta-a_{36}c_{23}&3(2 a_{34}+b_{35})\delta+a_{26}c_{23}
\end{smallmatrix}\right]=0
\label{eq2}\\
&T(6,5)=\tfrac{1}{9}\left( b_{26}^2+b_{36}^2+b_{25}^2+b_{35}^2+b_{25} a_{24}+b_{35} a_{34}-1\right)=0
\label{eq3}\\
&R(6,4)= \tfrac{1}{9}\left(a_{26}^2+a_{36}^2+a_{24}^2+ a_{34}^2+a_{24}b_{24}+a_{34} b_{35}+9\delta^2-1\right)=0. \label{eq4}
\end{align}

The rest of the proof will be divided into two cases, each case will lead us to one of the known examples.

We first assume that $c_{23}=0$. It follows from \eqref{eq1}, \eqref{eq2} and $\delta>0$ that
$$
b_{26}=b_{36}=a_{25}=a_{35}=a_{26}=a_{36}=0, b_{25}=-2
a_{24}, b_{35}=-2 a_{34}.
$$
From \eqref{eq3} and \eqref{eq4}, we have that $a_{24}^2+a_{34}^3=\tfrac{1}{2}$ and $\delta =\tfrac{\sqrt{6}}{6}$, hence  $\theta(A)$,  $\theta(B)$ and
$\theta(C)$ are respectively equal to
$$
\tfrac{1}{3}\left[\begin{smallmatrix}
&&&1&0&0\\&&&a_{24}&0&0\\&&&a_{34}&0&0\\1&a_{24}&a_{34}&&&\\0&0&0&&&\\0&0&0&&&\\
\end{smallmatrix}\right],
\tfrac{1}{3}\left[\begin{smallmatrix}
&&&0&1&0\\&&&0&-2a_{24}&0\\&&&0&-2a_{34}&0\\0&0&0&&&\\1&-2a_{24}&-2a_{34}&&&\\0&0&0&&&\\
\end{smallmatrix}\right],
\tfrac{1}{6}\left[\begin{smallmatrix}
&&&0&0&2\\&&&0&0&2a_{24}\\&&&0&0&2a_{34}\\
0&0&0&&&\sqrt{6}\\0&0&0&&0&\\2& 2a_{24} & 2a_{34} &-\sqrt{6}&& \\
\end{smallmatrix}\right],
$$
where $a_{24}^2+a_{34}^2=\tfrac{1}{2}$. Note that when $a_{24}=0$ and $a_{34}=\tfrac{\sqrt{2}}{2}$, we obtain $\theta(A_{M3}),\theta(B_{M3})$ and
$\theta(C_{M3})$ from Example \ref{M3}, and by acting with $h:=\left[\begin{smallmatrix}I&&\\&u&\\&&u^{-1}\end{smallmatrix}\right]\in U_0$ as in
\eqref{equiv1}, we have that $\theta(h_1A_{M3}h_1^{-1})$, $\theta(h_1B_{M3}h_1^{-1})$ and $\theta(h_1C_{M3}h_1^{-1})$ are given by
$$
\tfrac{1}{6}\left[\begin{smallmatrix}&&&2&0&0\\&&&\sqrt{2}s&0&0\\&&&\sqrt{2}c&0&0\\2&\sqrt{2}s&\sqrt{2}c&&&\\0&0&0&&&\\0&0&0&&&\\
\end{smallmatrix}\right],
\tfrac{1}{3}\left[\begin{smallmatrix}&&&0&1&0\\&&&0&-\sqrt{2}s&0\\&&&0&-\sqrt{2}c&0\\0&0&0&&&\\1&-\sqrt{2}s&-\sqrt{2}c&&&\\0&0&0&&&\\
\end{smallmatrix}\right],
\tfrac{1}{6}\left[\begin{smallmatrix}&&&0&0&2\\&&&0&0&\sqrt{2}s\\&&&0&0&\sqrt{2}c\\0&0&0&&&\sqrt{6}\\0&0&0&&0&\\2&\sqrt{2}s&\sqrt{2}c&-\sqrt{6}&&\\
\end{smallmatrix}\right].
$$
where $h_1:=\left[\begin{smallmatrix}u&\\&u^{-1}\end{smallmatrix}\right]$ and
$\theta(h_1)=
\left[\begin{smallmatrix}
1&&&&&\\&c&s&&&\\&-s&c&&&\\&&&1&&\\&&&&1&\\&&&&&1
\end{smallmatrix}\right]$, $c^2+s^2=1$.  This implies that we have covered
all the examples with $c_{23}=0$.  In other words, if $c_{23}=0$, then the ERP $G_2$-structure $(A_1,A,B,C)$ is equivariantly equivalent to $\mu_{M3}$.

Suppose now that $c_{23}\neq 0$.  Equation \eqref{eq1} implies that $ b_{25}=b_{35}=0$ and
$$
a_{25}=\tfrac{3b_{36}\delta }{c_{23}},\quad a_{35}=-\tfrac{3b_{26}\delta }{c_{23}},\quad b_{26}\left(c_{23}^2-9\delta^2\right)=0, \quad
b_{36}\left(c_{23}^2-9\delta^2\right)=0.
$$
By \eqref{eq3}, $b_{26}^2+b_{36}^2=1$, thus $c_{23}^2=9\delta^2$.  On the other hand, from \eqref{eq2} we have that
$$
a_{26}=-\tfrac{6a_{34}\delta}{c_{23}},\quad a_{36}=\tfrac{6a_{24}\delta}{c_{23}} ,\quad a_{24}\left(c_{23}^2-36\delta^2\right)=0,\quad
a_{34}\left(c_{23}^2-36\delta^2\right)=0.
$$
Hence $a_{24}=a_{34}=0$ and so \eqref{eq4} gives that $9\delta^2=1$. The matrices $\theta(A)$, $\theta(B)$ and $\theta(C)$ therefore read as follows
$$
\tfrac{1}{3}\left[\begin{smallmatrix} &&&1&0&0\\&&&0&c_{23} b_{36}&0\\&&&0&-c_{23} b_{26}&0\\1&0&0&&&\\0&c_{23} b_{36}&-c_{23}
b_{26}&&&\\0&0&0&&&\\\end{smallmatrix}\right], \tfrac{1}{3}\left[\begin{smallmatrix}&&&0&1&0\\&&&c_{23} b_{36}&0&b_{26}\\&&&-c_{23}
b_{26}&0&b_{36}\\0&c_{23} b_{36}&-c_{23} b_{26}&&&\\1&0&0&&&\\0&b_{26}&b_{36}&&&\\\end{smallmatrix}\right],
\tfrac{1}{3}\left[\begin{smallmatrix}&&&0&0&1\\&&c_{23}&0&b_{26}&0\\&-c_{23}&&0&b_{36}&0\\0&0&0&&&1\\0&b_{26}&b_{36}&&0& \\1& 0& 0&-1& &\\
\end{smallmatrix}\right],
$$
where $a_{25}^2+a_{35}^2=1$, $\delta=\tfrac{1}{3}$ and $c_{23}=\pm1$.  Note that when $a_{25}=-1$, $a_{35}=0$ and $c_{23}=-1$, we obtain $\theta(A_{M2}),\theta(B_{M2}),\theta(C_{M2})$ from Example \ref{M2}, and if we act with
$h\in U_0$ as in the case when $c_{23}=0$, then $\theta(h_1A_{M2}h_1^{-1})$, $\theta(h_1B_{M2}h_1^{-1})$ and $\theta(h_1C_{M2}h_1^{-1})$ are
respectively as follows
$$
\tfrac{1}{3}\left[\begin{smallmatrix}&&&1&0&0\\&&&0&-c&0\\&&&0&s&0\\1&0&0&&&\\0&-c&s&&&\\0&0&0&&&\\
\end{smallmatrix}\right],\quad
\tfrac{1}{3}\left[\begin{smallmatrix}&&&0&1&0\\&&&-c&0&s\\&&&s&0&c\\0&-c&s&&&\\1&0&0&&&\\0&s&c&&&\\
\end{smallmatrix}\right], \quad
\tfrac{1}{3}\left[\begin{smallmatrix}&&&0&0&1\\&&-1&0&s&0\\&1&&0&c&0\\0&0&0&&&1\\0&s&c&&0&\\1&0&0&-1&&\\
\end{smallmatrix}\right],
$$
The above matrices are therefore covered only when $c_{23}=-1$.  To reach the cases with $c_{23}=1$, we have to act with
$$
\bar{h}:=\left[\begin{smallmatrix}1&&&\\&-I&&\\&&u&\\&&&-u^{-1}\end{smallmatrix}\right]\in U_0, \quad
\bar{h}_1:=\left[\begin{smallmatrix}u&\\&-u^{-1}\end{smallmatrix}\right], \quad
\theta(\bar{h}_1)=\left[\begin{smallmatrix}1&&&&&&\\&c&s&&&&\\&-s&c&&&&\\&&&&1&&\\&&&&&-1&\\&&&&&&-1\\\end{smallmatrix}\right], \quad c^2+s^2=1,
$$
to obtain that $\theta(\bar{h}_1A_{M2}\bar{h}_1^{-1})$, $\theta(\bar{h}_1B_{M2}\bar{h}_1^{-1})$ and $\theta(\bar{h}_1C_{M2}\bar{h}_1^{-1})$ are
respectively equal to
$$
\tfrac{1}{3}\left[\begin{smallmatrix}&&&1&0&0\\&&&0&c&0\\&&&0&-s&0\\1&0&0&&&\\0&c&-s&&&\\0&0&0&&&\\
\end{smallmatrix}\right],\quad
-\tfrac{1}{3}\left[\begin{smallmatrix}&&&0&1&0\\&&&c&0&s\\&&&-s&0&c\\0&c&-s&&&\\1&0&0&&&\\0&s&c&&&\\
\end{smallmatrix}\right],\quad
-\tfrac{1}{3}\left[\begin{smallmatrix}&&&0&0&1\\&&1&0&s&0\\&-1&&0&c&0\\0&0&0&&&1\\0&s&c&&0&\\1&0&0&-1&&\\
\end{smallmatrix}\right].
$$
Hence if $c_{23}\neq 0$, then $(A_1,A,B,C)$ is equivariant equivalent to $\mu_{M2}$, which completes the proof of the proposition.
\end{proof}

\section{Case $\dim{\ngo}=6$}\label{n6-sec}

Just as in Section \ref{n5-sec}, we prove here that the two known examples of ERP $G_2$-structures on Lie groups are actually the only ones with a $6$-dimensional nilradical, up to equivalence and scaling.  However, up to equivariant equivalence, a 2-parameter family around one of these examples must be added to complete the classification.

\begin{example}\label{B} (see \cite[Example 5.7]{ERP} and \cite[Example 1]{Bry})
Let $\mu_{B}$ be the ERP $G_2$-structure given by
$$
(A_1)_B=\tfrac{1}{3}\left[\begin{smallmatrix} 1&\\ &1\end{smallmatrix}\right], \quad A_B=\tfrac{1}{6}\left[\begin{smallmatrix} -1&&& \\ &-1&& \\
&&1&\\ &&&1
\end{smallmatrix}\right],\quad
B_B=\tfrac{1}{3}\left[\begin{smallmatrix} &&&0 \\ &&0& \\ &1&&\\ 1&&&
\end{smallmatrix}\right], \quad
C_B=\tfrac{1}{3}\left[\begin{smallmatrix} &&0&0 \\ &&0&0 \\ 1&0&&\\ 0&-1&&
\end{smallmatrix}\right],
$$
from which follows that
$$
\theta(A_B)=\tfrac{1}{3}
\left[\begin{smallmatrix}
&&&1&0&0\\&&&0&0&0\\&&&0&0&0\\
1&0&0&&&\\0&0&0&&&\\0&0&0&&&
\end{smallmatrix}\right], \quad  \theta(B_B) =
\tfrac{1}{3}\left[\begin{smallmatrix}
&&&0&1&0\\&&&0&0&0\\&&&0&0&0\\
0&0&0&&1&\\1&0&0&-1&&\\0&0&0&&&
\end{smallmatrix}\right], \quad   \theta(C_B) =
\tfrac{1}{3}\left[\begin{smallmatrix}
&&&0&0&1\\&&&0&0&0\\&&&0&0&0\\
0&0&0&&&1\\0&0&0&&0&\\1&0&0&-1&&
\end{smallmatrix}\right].
$$
\end{example}

\begin{example}\label{rt}\cite[Example 6.4]{FinRff2}
For each pair $r,t\in\RR$, let $\mu_{rt}$ denote the $G_2$-structure given by
$$
(A_1)_{rt}=\tfrac{1}{3}\left[\begin{smallmatrix} 1&-r\\ r&1\end{smallmatrix}\right], \quad  A_{rt}=\tfrac{1}{6}\left[\begin{smallmatrix} -1&-2t&& \\ 2t&-1&& \\
&&1&2(r+t)\\ &&-2(r+t)&1
\end{smallmatrix}\right],
$$
$$
B_{rt}=\tfrac{1}{3}\left[\begin{smallmatrix} &&0&0 \\ &&0&0 \\ 0&1&&\\ 1&0&&
\end{smallmatrix}\right], \quad
C_{rt}=\tfrac{1}{3}\left[\begin{smallmatrix} &&0&0 \\ &&0&0 \\ 1&0&&\\ 0&-1&&
\end{smallmatrix}\right].
$$
Note that when $r=t=0$ one recovers the example $\mu_B$ above.  It follows from \cite[Proposition 3.5]{ERP} that $(G_{\mu_{rt}},\vp)$ is equivalent to $(G_{\mu_{B}},\vp)$ and consequently ERP for all $r,t\in\RR$.  Furthermore, since
$$
\RR^*\Spec(\ad_{\mu_{rt}}{e_7}|_{\hg}) = \RR^*\{  \tfrac{1}{3}+\im r, \tfrac{1}{3}+\im r, -\tfrac{1}{6}+\im t, -\tfrac{1}{6}+\im t, \tfrac{1}{6}-\im (r+t), \tfrac{1}{6}-\im (r+t)\},
$$
is an isomorphism invariant, one obtains that the family of Lie algebras $\{\mu_{rt}:r,t\in\RR\}$ is pairwise non-isomorphic.  In particular, the family of $G_2$-structures $\{ (G_{\mu_{rt}},\vp):r,t\in\RR\}$ is pairwise non-equivariantly equivalent.

It is straightforward to check that $\theta(A_{rt})$, $\theta(B_{rt})$ and $\theta(C_{rt})$ are respectively given by
$$
\tfrac{1}{3}\left[\begin{smallmatrix}&&&1&0&0\\&&r+2t&0&0&0\\&-r-2t&&0&0&0\\1&0&0&&&\\0&0&0&&&-r\\0&0&0&&r&
\end{smallmatrix}\right],   \quad \tfrac{1}{3}\left[\begin{smallmatrix}
&&&0&1&0\\&&&0&0&0\\&&&0&0&0\\0&0&0&&1&\\1&0&0&-1&&\\0&0&0&&&
\end{smallmatrix}\right],  \quad  \tfrac{1}{3}\left[\begin{smallmatrix}
&&&0&0&1\\&&&0&0&0\\&&&0&0&0\\0&0&0&&&1\\0&0&0&&0&\\1&0&0&-1&&
\end{smallmatrix}\right].
$$
\end{example}

\begin{example}\label{M1} \cite[Example 5.8]{ERP})
Consider the ERP $G_2$-structure $\mu_{M1}$, where
$$
(A_1)_{M1}=\tfrac{1}{30}\left[\begin{smallmatrix}
 \sqrt{30} & 0 \\  0 & 2\sqrt{30}
\end{smallmatrix}\right], \qquad
A_{M1}=\tfrac{1}{60}\left[\begin{smallmatrix}
 -10- \sqrt{30} & 0 & -2\sqrt{5} & 0 \\ 0 & -10 + \sqrt{30} & 0 & -2\sqrt{5} \\ -2\sqrt{5} & 0 & 10- \sqrt{30} & 0 \\ 0 & -2\sqrt{5} & 0 & 10+ \sqrt{30} \\
\end{smallmatrix}\right],
$$
$$
B_{M1}=\tfrac{1}{30}\left[\begin{smallmatrix}
 0 & - \sqrt{5} & 0 & 5- \sqrt{30} \\ 5\sqrt{5} & 0 & 5 & 0 \\ 0 & 5+\sqrt{30} & 0 & \sqrt{5} \\ 5  & 0 & -5 \sqrt{5} & 0 \\
\end{smallmatrix}\right], \qquad
C_{M1}=\tfrac{1}{30}\left[\begin{smallmatrix}
 -\sqrt{5} & 0 & 5- \sqrt{30} & 0 \\ 0 &  \sqrt{5} & 0 & -5+ \sqrt{30} \\5+\sqrt{30} & 0 & \sqrt{5} & 0 \\ 0 & -5-\sqrt{30} & 0 & -\sqrt{5} \\
\end{smallmatrix}\right].
$$
It easily follows that $\theta(A_{M1})$, $\theta(B_{M1})$ and $\theta(C_{M1})$ are respectively equal to
$$
\tfrac{1}{30}\left[\begin{smallmatrix} &&&10&0&0\\&&&0&-\sqrt{30}&0\\&&&2\sqrt{5}&0&0\\10&0&2\sqrt{5}&&&\\0&-\sqrt{30}&0&&&\\0&0&0&&&
\end{smallmatrix}\right],  \tfrac{1}{30}\left[\begin{smallmatrix}
&&&0&10&0\\&&6\sqrt{5}&-\sqrt{30}&0&0\\&-6\sqrt{5}&&0&-4\sqrt{5}&0\\0&-\sqrt{30}&0&&\sqrt{30}&\\10&0&-4\sqrt{5}&-\sqrt{30}&&\\0&0&0&&&
\end{smallmatrix}\right],
\tfrac{1}{15}\left[\begin{smallmatrix} &&&0&0&5\\&&&0&0&0\\&&&0&0&\sqrt{5}\\0&0&0&&&\sqrt{30}\\0&0&0&&0&\\5&0&\sqrt{5}&-\sqrt{30}&&
\end{smallmatrix}\right].
$$
\end{example}

We are now ready to prove the main result in this section.

\begin{proposition}\label{n6-clasif}
Any Lie group endowed with a left-invariant ERP $G_2$-structure and having a $6$-dimensional nilradical is equivariantly equivalent (up to scaling) to either $(G_{\mu_{M1}},\vp)$ or some $(G_{\mu_{rt}},\vp)$, $r,t\in\RR$.
\end{proposition}

\begin{proof}
It follows from Section \ref{preli} that we can assume that the Lie bracket $\mu$ of an ERP $G_2$-structure $(G,\vp)$ is given by $\mu=(A_1,A,B,C)$, where $A_1$ and $A$ are normal and $B$ and $C$ nilpotent.

We first consider the case when $A_1$ and $A$ are symmetric, so it can be assumed that $A_1=\left[\begin{smallmatrix}\alpha&0\\ 0&\delta\end{smallmatrix}\right]$, where $\alpha+\delta>0$ and $\delta\geq \alpha$ (see Section \ref{preli}).  In much the same way as in the proof of Proposition \ref{n5-clasif}, we first compute the form of $\theta(A)$, $\theta(B)$ and $\theta(C)$ by applying the conditions provided by Theorem \ref{main}, (i) and (ii).  From the nullity of the first column of each of the Jacobi condition matrices
$$
R:=[\theta(A),\theta(B)]-\alpha \theta(B)=0, \quad S:=[\theta(A),\theta(C)]-\delta \theta(C)=0,\quad T:=[\theta(B),\theta(C)]=0,
$$
we obtain that
\begin{equation}\label{n6jac}
\begin{array}{l}
c_{25}= b_{26}, \quad c_{35}= b_{36}, \quad c_{24}= a_{26}, \quad c_{34}= a_{36}, \\
b_{24}= a_{25}, \quad b_{34}= a_{35}, \quad b_{45}= 3\alpha, \quad c_{45}= b_{46}= 0,
\end{array}
\end{equation}
and so the matrices $3\theta(A)$, $3\theta(B)$ and $3\theta(C)$ are respectively given by
$$
\left[\begin{smallmatrix}
&&&1&0&0\\&&&a_{24}&a_{25}&a_{26}\\&&&a_{34}&a_{35}&a_{36}\\1&a_{24}&a_{34}&&&\\0&a_{25}&a_{35}&&&\\0&a_{26}&a_{36}&&&
\end{smallmatrix}\right],
\left[\begin{smallmatrix}
&&&0&1&0\\&&b_{23}&a_{25}&b_{25}&b_{26}\\&-b_{23}&&a_{35}&b_{35}&b_{36}\\0&a_{25}&a_{35}&&3\alpha&\\1&b_{25}&b_{35}&-3\alpha&&\\0&b_{26}&b_{36}&&&
\end{smallmatrix}\right],
\left[\begin{smallmatrix}
&&&0&0&1\\&&c_{23}&a_{26}&b_{26}&-(a_{24}+b_{25})\\&-c_{23}&&a_{36}&b_{36}&-(a_{34}+b_{35})\\0&a_{26}&a_{36}&&&3\delta\\0&b_{26}&b_{36}&&0&\\1&-(a_{24}+b_{25})&-(a_{34}+b_{35})&-3\delta&&
\end{smallmatrix}\right].
$$
Note that $b_{23}$ and $\alpha$ can not simultaneously vanish, since in that case $\theta(B)$ and so $B$ would be symmetric, which is a contradiction as $B$ is nilpotent.

We write below the remaining expressions given by the nullity of $R,S,T$ that are needed in the proof:
\begin{align}
&S[4..5,2..3]=\tfrac{1}{9}\left[\begin{smallmatrix}-6\delta a_{26}-a_{34}c_{23} & -6\delta a_{36}+a_{24}c_{23}=0\\
-3\delta b_{26}-a_{35}c_{23} &-3\delta b_{36}+a_{25}c_{23}
\end{smallmatrix}\right]=0,\label{eq5}\\
&R[4..5,2..3]=\tfrac{1}{9}\left[\begin{smallmatrix}-6\alpha a_{25}-a_{34}b_{23} & -6\alpha a_{35}+a_{24}b_{23}\\
3\alpha (a_{24}-b_{25})-a_{35}b_{23} &3\alpha (a_{34}-b_{35})+a_{25}b_{23}
\end{smallmatrix}\right]=0,\label{eq6}\\
&R[6,4]=\tfrac{1}{9}(a_{25}a_{26}+a_{35}a_{36}-a_{24}b_{26}-a_{34}b_{36})=0,\label{eq7}\\
&R[3,2]=\tfrac{1}{9}(a_{34}a_{25}+a_{35}b_{25}+a_{36}b_{26}-a_{24}a_{35}-a_{25}b_{35}-a_{26}b_{36}+3\alpha b_{23})=0,\label{eq8} \\
&R[5,4]=\tfrac{1}{9}\left(a_{25}^2+a_{35}^2-b_{25}a_{24}-b_{35}a_{34}-1+9\right) \alpha^2=0,\label{eq9}\\
&T[5,2..3]=\tfrac{1}{9}\left[\begin{smallmatrix}b_{23} b_{36}-b_{35} c_{23}-3 a_{26} \alpha & -b_{23} b_{26}+b_{25} c_{23}-3 a_{36}
\alpha\end{smallmatrix}\right]=0,\label{eq10}\\
&S[6,2..3]=\tfrac{1}{9}\left[\begin{smallmatrix}3\delta(2a_{24}+b_{25})-a_{36}c_{23} & 3\delta(2a_{34}+b_{35})+a_{26}c_{23}\end{smallmatrix}\right]=0,\label{eq11}\\
&T[6,2..3]=\tfrac{1}{9}\left[\begin{smallmatrix}-b_{23}(a_{34}+b_{35})+3 a_{25} \delta-b_{36} c_{23} &  b_{23} (a_{24}+b_{25})+3 a_{35} \delta+b_{26}
c_{23}\end{smallmatrix}\right]=0.\label{eq12}
\end{align}
The proof is divided into two steps, depending on the value of $c_{23}$. The case when $c_{23}\neq 0$ leads to a contradiction, and when $c_{23}=0$, the value of $b_{23}$ determines if the ERP $G_2$-structure $\mu$ is either equivariant equivalent to $\mu_B$ or $\mu_{M1}$.

We first assume $c_{23}\neq 0$, so from \eqref{eq5},
$$
a_{34}=-\frac{6 \delta}{c_{23}} a_{26},\quad a_{24}=\frac{6 \delta}{c_{23}}a_{36},\quad a_{35}=-\frac{3 \delta}{c_{23}}b_{26},\quad a_{25}=\frac{3
\delta}{c_{23}}b_{36}.
$$
Replacing these values in \eqref{eq6} and \eqref{eq7}, and keeping in mind that $b_{23}$ and $\alpha$ can not simultaneously vanish, we obtain that
$a_{26}=a_{36}=b_{26}=b_{36}=0$. Consequently, we have that $0=S[3,2]=\tfrac{1}{3}\delta c_{23}$, which is a contradiction.

Assume now that $c_{23}=0$, thus $a_{26}=a_{36}=b_{26}=b_{36}=0$ by \eqref{eq5} and from \eqref{eq11}, we therefore obtain that $b_{25}=-2a_{24}$ and $b_{35}=-2a_{34}$.  If $b_{23}=0$, then $\alpha\neq 0$ and it follows from \eqref{eq6} that $a_{24}=a_{34}=a_{25}=a_{35}=0$ and
$$
0=S[6,4]=-\tfrac{1}{9}+\delta^2, \qquad 0=T[6,5]=-\tfrac{1}{9}+\alpha \delta.
$$
Hence $\delta=\alpha=\tfrac{1}{3}$ and so $\theta(A)=\theta(A_B)$, $\theta(B)=\theta(B_B)$, $\theta(C)=\theta(C_B)$, that is, we obtain that $\mu=\mu_B$.
On the contrary, if $b_{23}\neq 0$, then by \eqref{eq6},
$$
a_{25}=-\tfrac{9\alpha}{b_{23}}a_{34},\quad a_{35}=\tfrac{9\alpha}{b_{23}}a_{24},
$$
therefore from \eqref{eq6} and \eqref{eq8}, it follows that
$$
a_{24}^2+a_{34} = 6\alpha^2, \quad b_{23}^2=54\alpha^2,
$$
but \eqref{eq9} and the fact that $\alpha+\delta>0$ imply that $\alpha^2=\tfrac{1}{30}$, and by \eqref{eq12}
$$
\alpha=\tfrac{\sqrt{30}}{30}, \quad \delta=2\tfrac{\sqrt{30}}{30}, \quad  b_{23}=\tfrac{3}{5}\epsilon\sqrt{5},\quad \epsilon:=\pm1,
$$
from which follows that
$$
\theta(A)=\tfrac{1}{30}
\left[\begin{smallmatrix}
&&&10&0&0\\&&&10a_{24}&-5\sqrt{6}\epsilon a_{34}&0\\&&&10a_{34}&5\sqrt{6}\epsilon a_{24}&0\\10&10a_{24}&10a_{34}&&&\\0&-5\sqrt{6}a_{34}&5\sqrt{6}\epsilon
a_{24}&&&\\ 0&0&0&&&
\end{smallmatrix}\right],
$$
$$
\theta(B)=\tfrac{1}{30}\left[\begin{smallmatrix} &&&0&10&0\\&&6\sqrt{5}\epsilon&-5\sqrt{6}\epsilon a_{34}&-20
a_{24}&0\\&-6\sqrt{5}\epsilon&&5\sqrt{6}\epsilon
a_{24}&-20a_{34}&0\\0&-5\sqrt{6}\epsilon a_{34}&5\sqrt{6}\epsilon a_{24}&&\sqrt{30}&\\10&-20a_{24}&-20a_{34}&-\sqrt{30}&&\\
0&0&0&&&\end{smallmatrix}\right], \quad \theta(C)=
\tfrac{1}{15}\left[\begin{smallmatrix}&&&0&0&5\\&&&0&0&5a_{24}\\&&&0&0&5a_{34}\\0&0&0&&&\sqrt{30}\\0&0&0&&0&\\
5&5a_{24}&5a_{34}&-\sqrt{30}&&
\end{smallmatrix}\right].
$$
Let us denote by $\theta(A)(\epsilon,a_{24},a_{34})$, $\theta(B)(\epsilon,a_{24},a_{34})$ and $\theta(C)(\epsilon,a_{24},a_{34})$ the above matrices.  Note
that
$$
\theta(A)\left(1,0,\tfrac{1}{\sqrt{5}}\right)=\theta(A_{M1}), \quad\theta(B)\left(1,0,\tfrac{1}{\sqrt{5}}\right)=\theta(B_{M1}),
\quad\theta(C)\left(1,0,\tfrac{1}{\sqrt{5}}\right)=\theta(C_{M1}).
$$
If we act on $\mu_{M1}$ with $h\in U_0$ as in \eqref{equiv1}, then we obtain the equivariantly equivalent $G_2$-structure for which the matrices $\theta(h_1A_{M1}h_1^{-1})$, $\theta(h_1B_{M1}h_1^{-1})$ and $\theta(h_1C_{M1}h_1^{-1})$ are respectively given by
$$
\theta(A)\left(1,\tfrac{s}{\sqrt{5}},\tfrac{c}{\sqrt{5}}\right),\quad  \theta(B)\left(1,\tfrac{s}{\sqrt{5}},\tfrac{c}{\sqrt{5}}\right), \quad
\theta(C)\left(1,\tfrac{s}{\sqrt{5}},\tfrac{c}{\sqrt{5}}\right).
$$
On the other hand, if we act on $\mu_{M1}$ with $\bar{h}\in U_0$ as in \eqref{equiv1}, then we obtain
$$
\theta(A)\left(-1,\tfrac{s}{\sqrt{5}},\tfrac{c}{\sqrt{5}}\right),\quad -\theta(B)\left(-1,\tfrac{s}{\sqrt{5}},\tfrac{c}{\sqrt{5}}\right), \quad -\theta(C)\left(-1,\tfrac{s}{\sqrt{5}},\tfrac{c}{\sqrt{5}}\right).
$$
Hence, we reach all the above matrices, from which follows that if $c_{23}=0$ and $b_{23}\neq0$, then $\mu$ is equivariantly equivalent to $\mu_{M1}$.

Secondly, we consider the case when $\mu=(A_1,A,B,C)$, where $A_1$ and $A$ are both normal but at least one of them is not symmetric, so $A_1=\left[\begin{smallmatrix}\alpha&\beta\\ -\beta&\delta\end{smallmatrix}\right]$, where either $\beta=0$ or $\alpha=\delta$ (see Section \ref{preli}).  In this case, by using Theorem \ref{main}, (i) and (ii) and the equations provided by only the first columns of the following matrix equations determined by the Jacobi condition,
$$
[\theta(A),\theta(B)]=\alpha \theta(B)-\beta \theta(C), \quad [\theta(A),\theta(C)]=\beta \theta(B)
+\delta \theta(C), \quad [\theta(B),\theta(C)]=0,
$$
one obtains that all the conditions in \eqref{n6jac} hold plus the following extra ones,
$$
a_{45}=0, \quad a_{46}=0, \quad a_{56}=3\beta,
$$
and so the matrices $3\theta(A)$, $3\theta(B)$ and $3\theta(C)$ are respectively given by
$$
\left[\begin{smallmatrix}
&&&1&0&0\\&&a_{23}&a_{24}&a_{25}&a_{26}\\&-a_{23}&&a_{34}&a_{35}&a_{36}\\1&a_{24}&a_{34}&&&\\0&a_{25}&a_{35}&&&3\beta\\0&a_{26}&a_{36}&&-3\beta&
\end{smallmatrix}\right],
\left[\begin{smallmatrix}
&&&0&1&0\\&&b_{23}&a_{25}&b_{25}&b_{26}\\&-b_{23}&&a_{35}&b_{35}&b_{36}\\0&a_{25}&a_{35}&&3\alpha&\\1&b_{25}&b_{35}&-3\alpha&&\\0&b_{26}&b_{36}0&&&
\end{smallmatrix}\right],
\left[\begin{smallmatrix}
&&&0&0&1\\&&c_{23}&a_{26}&b_{26}&-(a_{24}+b_{25})\\&-c_{23}&&a_{36}&b_{36}&-(a_{34}+b_{35})\\0&a_{26}&a_{36}&&&3\delta\\0&b_{26}&b_{36}&&0&\\1&-(a_{24}+b_{25})&-(a_{34}+b_{35})&-3\delta&&
\end{smallmatrix}\right].
$$
From the following fact:
$$
\theta(E)=
\left[\begin{smallmatrix}
0&&&&&\\ &0&a-b&&&\\
&b-a&0&&&\\ &&&0&&\\ &&&&0&a+b\\ &&&&-(a+b)&0
\end{smallmatrix}\right], \quad\mbox{for any} \quad
E=\left[\begin{smallmatrix}
0&-a&&\\ a&0&&\\ &&0&-b\\ &&b&0
\end{smallmatrix}\right], \quad a,b\in\RR,
$$
we obtain that $\ad{e_7}|_{\hg}=\tilde{A}+D$, where $\tilde{A}$ and $D$ are respectively the symmetric and skew-symmetric parts, given by
$$
\tilde{A}:=
\left[\begin{smallmatrix}
\alpha&0&&&\\0&\delta&&&\\
&&&&\\
&&&S(A)&\\
&&&&
\end{smallmatrix}\right], \qquad D:=
\tfrac{1}{6}\left[\begin{smallmatrix}
0&6\beta&&&&\\ -6\beta&0&&&&\\
&&0&-a_{23}-3\beta&&\\
&&a_{23}+3\beta&0&&\\
&&&&0&a_{23}-3\beta\\
&&&&-a_{23}+3\beta&0
\end{smallmatrix}\right],
$$
and $S(A)$ denotes the symmetric part of $A$ (see \eqref{magic}).  Since $D\in\Der(\hg)\cap\sug(3)$ and commutes with $\ad{e_7}|_{\hg}$ as $A$ is normal, it follows from \cite[Proposition 3.5]{ERP} that $(A_1,A,B,C)$ is equivalent as a $G_2$-structure to $(S(A_1),S(A),B,C)$, which is therefore ERP.  By the case worked out above, $(S(A_1),S(A),B,C)$ must be precisely $\mu_B$, since the other possibility, $\mu_{M1}$ (up to equivariance equivalence), does not admit any derivation in $\sug(3)$ (see Section \ref{sym-sec}).  This implies that $\mu=\mu_{rt}$, where $r=-3\beta$, $t=\unm a_{23}+\frac{3}{2}\beta$, completing the proof of the proposition.
\end{proof}

\begin{remark}
It can be directly shown that, up to equivalence, a Lie group with an ERP $G_2$-structure has to be completely solvable by only using the structure results obtained in \cite{ERP} and the fact that they are solvsolitons (see \cite[Corollary 1.2]{ERP}).  Indeed, it follows from Theorem \ref{main}, (ii) that $\theta(A)$, $\theta(B)$ and $\theta(C)$ have respectively the form
$$
\left[\begin{smallmatrix}
0&0&0&1&0&0\\
0&0&a_{23}&a_{24}&a_{25}&a_{26}\\
0&-a_{23}&0&a_{34}&a_{35}&a_{36}\\
1&a_{24}&a_{34}&0&a_{45}&a_{46}\\
0&a_{25}&a_{35}&-a_{45}&0&a_{56}\\
0&a_{26}&a_{36}&-a_{46}&-a_{56}&0
\end{smallmatrix}\right],\quad
\left[\begin{smallmatrix}
0&0&0&0&1&0\\
0&0&b_{23}&b_{24}&b_{25}&b_{26}\\
0&-b_{23}&0&b_{34}&b_{35}&b_{36}\\
0&b_{24}&b_{34}&0&b_{45}&b_{46}\\
1&b_{25}&b_{35}&-b_{45}&0&b_{56}\\
0&b_{26}&b_{36}&-b_{46}&-b_{56}&0
\end{smallmatrix}\right], \quad
\left[\begin{smallmatrix}
0&0&0&0&0&1\\
0&0&c_{23}&c_{24}&c_{25}&c_{26}\\
0&-c_{23}&0&c_{34}&c_{35}&c_{36}\\
0&c_{24}&c_{34}&0&c_{45}&c_{46}\\
0&c_{25}&c_{35}&-c_{45}&0&c_{56}\\
1&c_{26}&c_{36}&-c_{46}&-c_{56}&0
\end{smallmatrix}\right],
$$
and according to \cite[Theorem 4.8]{solvsolitons}, $\theta(A)$ is always normal, in addition $\theta(B)$ is normal if $\dim{\ngo}=4,5$ and $\theta(C)$ must also be normal if $\dim{\ngo}=4$.  A very easy computation therefore gives that $a_{45}=a_{46}=0$, and in addition $b_{45}=b_{56}=0$ and $c_{46}=c_{56}=0$, respectively.  It is now straightforward to show that they must be symmetric in order to commute (which is mandatory by \cite[Theorem 4.8]{solvsolitons}) and satisfy the Jacobi condition, and so $A$ and $B$ are symmetric if $\dim{\ngo}=5$ and they are all symmetric if $\dim{\ngo}=4$.  In particular, the corresponding solvable Lie groups are completely solvable.  In the case when $\dim{\ngo}=6$, one obtains that the ERP structure is equivalent to  one on the completely solvable Lie group obtained by replacing $A$ with the symmetric part $S(A)$ of $A$, as shown at the end of the proof of Proposition \ref{n6-clasif}.
\end{remark}

\section{Symmetries}\label{sym-sec}

We aim in this section to provide some insight on the symmetries of each of the six ERP $G_2$-structures obtained in the classification Theorem \ref{clasif}.

The isometry group of a left-invariant Riemannian metric on an $n$-dimensional Lie group can be very tricky to compute, even in the completely solvable case.  It is not hard to see, however, that $\Iso(G_\mu,\ip)=KG_\mu$, where $K:=\Iso(G_\mu,\ip)_e$ is the isotropy subgroup at the identity and $G_\mu$ also denotes the subgroup of left-translations. The following conditions are also easily seen to be equivalent:
\begin{itemize}
\item[(i)] $G_\mu$ is normal in $\Iso(G_\mu,\ip)$.
\item[(ii)] $\Iso(G_\mu,\ip)=K\ltimes G_\mu$.
\item[(iii)] $K=\Aut(G_\mu)\cap\Iso(G_\mu,\ip)$, which is identified with the group
$$
\Aut(\mu)\cap\Or(n), \qquad \Or(n):=\Or(\ggo,\ip),
$$
of orthogonal automorphisms of the Lie algebra.
\end{itemize}
This is known to hold if $\mu$ is unimodular and completely solvable (see \cite{GrdWls}).  In any case, the subgroup of isometries of $(G_\mu,\ip)$ given by
$$
\left(\Aut(\mu)\cap\Or(n)\right)\ltimes G_\mu,
$$
is always present and less difficult to compute.

On the other hand, given a Lie group $(G_\mu,\vp)$ endowed with a left-invariant $G_2$-structure, we can also consider the subgroup of automorphisms of $(G_\mu,\vp)$ given by,
$$
(\Aut(\mu)\cap G_2)\ltimes G_\mu \subset \Aut(G_\mu,\vp)\subset\Iso(G_\mu,\ip),
$$
and since $G_2\subset\SO(7)$ ($\ip:=\ip_\vp$), we obtain that
\begin{equation}\label{AA}
\Aut(\mu)\cap G_2\subset\Aut(\mu)\cap\Or(7).
\end{equation}
Note that $(\Aut(\mu)\cap G_2)\ltimes G_\mu = \Aut(G_\mu,\vp)$ also holds in the unimodular completely solvable case.

In this light, we compute in what follows the two groups given in \eqref{AA} for each of the ERP structures appearing in Theorem \ref{clasif}.   It is straightforward to show that the Lie algebra $\Der(\mu)\cap\sog(7)$ of $\Aut(\mu)\cap\Or(7)$ is always zero except for $\mu_B$ and $\mu_{rt}$, where it coincides with $\ug_0$, the $2$-dimensional abelian Lie algebra of the Lie group $U_0\simeq S^1\times S^1$ given in \eqref{defU0}.  This implies that the groups in \eqref{AA} are all finite in the other four cases.

The results of our computations of $G_2$-automorphisms can be summarized as follows.  All the matrices below are written in terms of the basis $\{ e_7,e_3,e_4,e_1,e_2,e_5,e_6\}$.

\begin{enumerate}[{\small $\bullet$} ]
\item $\Aut(\mu_{B})\cap G_2 = \Aut(\mu_{rt})\cap G_2 = U_0 \simeq S^1\times S^1$.
\item[ ]
\item $\Aut(\mu_{M1})\cap G_2  = \la f_0\ra \simeq \ZZ_2$, where $f_0:=\Diag(1,1,1,-1,-1,-1,-1)$.
\item[ ]
\item $\Aut(\mu_{M2})\cap G_2 = \la f_0\ra \simeq  \ZZ_2$.
\item[ ]
\item $\Aut(\mu_{M3})\cap G_2 = \la f_1\ra\simeq  \ZZ_4$, where $f_1|_{\ggo_0}:=\Diag(1,-1,-1)$ and
$$
f_1|_{\ggo_1}:=\left[\begin{smallmatrix}
&0&-1&&\\
&1&0&&\\
&&&0&-1\\
&&&1&0
\end{smallmatrix}\right].
$$
\item[ ]
\item $\Aut(\mu_{J})\cap G_2 \simeq \Sl_2(\ZZ_3)$, the binary tetrahedral group of order $24$.  Indeed, it is easy to check that this group has order $24$, only one element of order $2$ and none of order $12$, a condition that characterizes $\Sl_2(\ZZ_3)$ among the groups of order $24$.
\end{enumerate}

\begin{remark}
From the original presentation of $(G_{\mu_B},\vp)$ as a homogeneous space $(G/K,\psi)$ endowed with a $G$-invariant $G_2$-structure $\psi$ given in \cite[Example 1]{Bry}, where
$$
G/K=\left(\Sl_2(\CC)\ltimes\CC^2\right)/\SU(2),
$$
we obtain that $\Aut(G_{\mu_B},\vp)$ actually contains a $6$-dimensional subgroup isomorphic to $\Sl_2(\CC)$ and that $\SU(2)\subset \Aut(G_{\mu_B},\vp)_e$.  Since $\Aut(\mu_B)\cap G_2 =S^1\times S^1$, this shows that there are indeed automorphisms of $(G_{\mu_B},\vp)$ which are not compositions of Lie group automorphisms and left-translations.  We do not know if this is also the case for the other examples, except for the unimodular case $\mu_J$ (cf. Corollary \ref{sym-J} below).
\end{remark}

On the other hand, concerning isometries we have obtained the following:

\begin{enumerate}[{\small $\bullet$} ]
\item $\Aut(\mu_{B})\cap \Or(7) =  \la f_2\ra\ltimes U_0\simeq \ZZ_2\ltimes (S^1\times S^1)$, where $f_2:=\Diag(1,1,-1,1,-1,-1,1)$.
\item[ ]
\item $\Aut(\mu_{rt})\cap\Or(7) =U_0\simeq S^1\times S^1$.
\item[ ]
\item $\Aut(\mu_{M1})\cap \Or(7)  = \la f_0,f_3\ra \simeq \ZZ_2\times\ZZ_2$, where $f_3:=\Diag(1,-1,1,-1,1,-1,1)$.
\item[ ]
\item $\Aut(\mu_{M2})\cap \Or(7)= \la f_0,f_2,f_4\ra \simeq  \ZZ_2\times\ZZ_2\times\ZZ_2$, where $f_4|_{\ggo_0}:=I$ and
$$
f_4|_{\ggo_1}:=\left[\begin{smallmatrix}
1&&&\\
&0&-1&\\
&-1&0&\\
&&&1
\end{smallmatrix}\right].
$$
\item[ ]
\item $\Aut(\mu_{M3})\cap \Or(7) = \la f_3,f_5,f_6\ra\simeq  D_4\times\ZZ^2$, where $D_4$ is the dihedral group of degree four (and order $8$), $f_5|_{\ggo_0}:=\Diag(1,-1,1)$, $f_6|_{\ggo_0}:=\Diag(1,1,-1)$ and
$$
f_5|_{\ggo_1}:=\frac{1}{3}\left[\begin{smallmatrix}
0& \sqrt{6}&0&\sqrt{3}\\
-\sqrt{6}&0&-\sqrt{3}&0\\
0&\sqrt{3}&0&-\sqrt{6}\\
-\sqrt{3}&0&\sqrt{6}&0
\end{smallmatrix}\right], \qquad
f_6|_{\ggo_1}:=\frac{1}{3}\left[\begin{smallmatrix}
-\sqrt{6}&0&-\sqrt{3}&0\\
0&-\sqrt{6}&0&-\sqrt{3}\\
-\sqrt{3}&0&\sqrt{6}&0\\
0&-\sqrt{3}&0&\sqrt{6}
\end{smallmatrix}\right].
$$
Indeed, the generators satisfy the following relations:
$$
f_3^2=f_5^4=(f_3f_5)^2=e, \quad f_3f_6=f_6f_3, \quad f_5f_6=f_6f_5.
$$
\item $\Aut(\mu_{J})\cap \Or(7) \simeq S_4\ltimes\ZZ_2^4$, where $S_4$ is the symmetric group of degree four (and order $24$), since it is isomorphic to the centralizer in $\Or(4)$ of  the maximal torus $\la A,B,C\ra$ of $\slg_4(\RR)$ (see \eqref{Jdef}).
\end{enumerate}

\begin{corollary}\label{sym-J}
The groups of automorphisms and isometries of the $G_2$-structure $(G_{\mu_J},\vp)$ are respectively given by
$$
\Aut(G_{\mu_J},\vp)= \Sl_2(\ZZ_3)\ltimes G_{\mu_J}, \qquad \Iso(G_{\mu_J},\ip)=\left(S_4\ltimes\ZZ_2^4\right)\ltimes G_{\mu_J}.
$$
\end{corollary}


\begin{thebibliography}{MMM}

\bibitem[A]{Alk} {\sc D. Alekseevskii}, Conjugacy of polar factorizations of Lie groups, {\it Mat. Sb.} {\bf 84} (1971), 14-26; {\it English translation}: {\it Math. USSR-Sb.} {\bf 13} (1971), 12-24.

\bibitem[AK]{AlkKml} {\sc D. Alekseevskii, B. Kimel'fel'd}, Structure of homogeneous Riemannian spaces with zero Ricci curvature, {\it Funktional Anal. i Prilozen} {\bf 9} (1975), 5-11 (English translation: {\it Functional Anal. Appl.} {\bf 9} (1975), 97-102).

\bibitem[AL]{ArrLfn} {\sc R. Arroyo, R. Lafuente}, The Alekseevskii conjecture in low dimensions, {\it Math. Annalen} {\bf 367} (2017), 283-309.

\bibitem[Ba]{Bll} {\sc G. Ball}, Seven-Dimensional Geometries With Special Torsion, Ph.D. dissertation, Duke Univ..

\bibitem[BL]{BhmLfn} {\sc C. B\"ohm, R. Lafuente}, The Ricci flow on solvmanifolds of real type, {\it Adv. Math.} {\bf 352} (2019), 516-540.

\bibitem[B]{Bry} {\sc R. Bryant}, Some remarks on $G_2$-structures, Proc. G\"okova Geometry-Topology Conference (2005), 75-109.

\bibitem[CI1]{ClyIvn1} {\sc R. Cleyton, S. Ivanov}, On the geometry of closed $G_2$-structures,  {\it Comm. Math. Phys.} {\bf 270} (2007), 53-67.

\bibitem[CI2]{ClyIvn2} {\sc R. Cleyton, S. Ivanov}, Curvature decomposition of G2-manifolds. {\it J. Geom. Phys.} {\bf 58} (2008), 1429-1449.

\bibitem[FFM]{FrnFinMnr} {\sc M. Fern\'andez, A. Fino, V. Manero}, $G_2$-structures on Einstein solvmanifolds,  {\it Asian J. Math.} {\bf 19} (2015), 321-342.

\bibitem[FFR]{FrnFinRff} {\sc M. Fern\'andez, A. Fino, A. Raffero}, Locally conformal calibrated $G_2$-manifolds,  {\it Ann. Mat. Pura Appl.} {\bf 195} (2016), 1721-1736.

\bibitem[FR]{FinRff2} {\sc A. Fino, A. Raffero}, A class of eternal solutions to the $G_2$-Laplacian flow, preprint 2018 (arXiv).

\bibitem[GW]{GrdWls} {\sc C. S. Gordon, E. N. Wilson}, Isometry groups of Riemannian solvmanifolds, {\it Trans. Amer. Math. Soc.} {\bf 307} (1988), 245–-269.

\bibitem[L1]{solvsolitons} {\sc J. Lauret}, Ricci soliton solvmanifolds, {\it J. reine angew. Math.} {\bf 650} (2011), 1-21.

\bibitem[L2]{LF}  {\sc J. Lauret}, Laplacian flow of homogeneous $G_2$-structures and its solitons, {\it Proc. London Math. Soc.} {\bf 114} (2017), 527-560.

\bibitem[L3]{LS-ERP}  {\sc J. Lauret}, Laplacian solitons: Questions and homogeneous examples, {\it Diff. Geom. Appl.} {\bf 54} (2017), 345-360.

\bibitem[LN]{ERP}  {\sc J. Lauret, M. Nicolini}, Extremally Ricci pinched $G_2$-structures on Lie groups, {\it Comm. Anal. Geom.}, in press (arXiv).

\bibitem[Lo]{Lty} {\sc J. Lotay}, Geometric flows of $G_2$ structures, {\it Fields Institute Communications}, Springer, in press (arXiv).

\bibitem[LoW]{LtyWei} {\sc J. Lotay, Y. Wei}, Laplacian flow for closed $G_2$ structures: Shi-type estimates, uniqueness and compactness, {\it Geom. Funct. Anal.} {\bf 27} (2017), 165-233.

\bibitem[PR]{PdsRff} {\sc F. Podesta, A. Raffero}, On the automorphism group of a closed $G_2$-structure, {\it Quart. J. Math.}, in press.

\end{thebibliography}
\end{document}